\documentclass{article}
\usepackage{amsmath}
\usepackage{amsthm}
\usepackage{graphics}
\usepackage{epsfig}

\newtheorem{theorem}{Theorem}[section]
\newtheorem{cor}[theorem]{Corollary}
\newtheorem{lem}[theorem]{Lemma}
\newtheorem{prop}[theorem]{Proposition}

\theoremstyle{definition}
\newtheorem{defn}[theorem]{Definition}
\newtheorem{rem}[theorem]{Remark}

\newtheorem{example}[theorem]{Example}

\numberwithin{equation}{section}

\usepackage{amssymb}
\usepackage{amsfonts}
\usepackage{stmaryrd}
\usepackage{psfrag}
\usepackage{color}
\usepackage{url}
\usepackage{mathtools}

\newcommand{\N}{\mathbb{N}}
\newcommand{\R}{\mathbb{R}}
\newcommand{\T}{\mathbb{T}}
\newcommand{\Z}{\mathbb{Z}}
\newcommand{\Q}{\mathbb{Q}}
\newcommand{\C}{\mathbb{C}}

\newcommand{\mS}{\mathbb{S}}

\newcommand{\cB}{\mathcal{B}}

\newcommand{\cI}{\mathcal{I}}
\newcommand{\cJ}{\mathcal{J}}
\newcommand{\cN}{\mathcal{N}}
\newcommand{\cP}{\mathcal{P}}
\newcommand{\cO}{\mathcal{O}}

\newcommand{\cR}{\mathcal{R}}

\newcommand{\cL}{\mathcal{L}}
\newcommand{\cZ}{\mathcal{Z}}
\newcommand{\cW}{\mathcal{W}}
\newcommand{\cU}{\mathcal{U}}

\begin{document}

\title{Most linear flows on $\mathbb{R}^d$ are Benford}

\author{Arno Berger\\[2mm] Mathematical and
    Statistical Sciences\\University of Alberta\\Edmonton, Alberta,
  {\sc Canada}}

\maketitle
\begin{abstract}
\noindent
A necessary and sufficient condition (``exponential nonresonance'') is
  established for every signal obtained from a linear flow on $\mathbb{R}^d$
  by means of a linear observable to either vanish identically or else
  exhibit a strong form of Benford's Law (logarithmic distribution of
  significant digits). The result extends and unifies all previously
  known (sufficient) conditions. Exponential nonresonance is
  shown to be typical for linear flows, both from a topological and a
  measure-theoretical point of view.
\end{abstract}
\hspace*{6.6mm}{\small {\bf Keywords.} Benford function, (continuous) uniform distribution
mod $1$, linear}\\[-1mm]
\hspace*{25.5mm}{\small flow, observable, $\Q$-independence, (exponentially) nonresonant set.}

\noindent
\hspace*{6.6mm}{\small {\bf MSC2010.} 34A30, 37A05, 37A45, 11J71, 11K41, 62E20.}

\medskip

\section{Introduction}

Let $\phi$ be a flow on $X=\R^d$ endowed with the usual
topology, i.e., $\phi: \R \times X \to X$ is continuous, and
$\phi(0,x)=x$ as well as $\phi \bigl( s, \phi(t,x)\bigr)= \phi(s+t,x)$
for all $x\in X$ and $s,t\in \R$. Denoting the homeomorphism
$x\mapsto \phi(t,x )$ of $X$ simply by $\phi_t$ and the space of all
linear maps $A:X\to X$ by $\cL (X)$, as usual, call the flow
$\phi$ {\em linear\/} if each $\phi_t$ is linear, that is, $\phi_t \in
\cL(X)$ for every $t\in \R$. Given a linear flow $\phi$ on $X$, fix
any linear functional $H:\cL (X)\to \R$ and consider the function
$H(\phi_{\bullet})$. The main goal of this article is to completely
describe the distribution of numerical values  for the real-valued functions thus generated.

To see why this distribution may be of interest, recall that throughout
science and engineering, flows on the phase space $X=\R^d$ are often used to provide
models for real-worlds processes; e.g., see \cite{Amann}. From a scientist's or
engineer's perspective, it may not be desirable or even possible to
observe a flow $\phi$ in its entirety, especially if $d$ is large. Rather, what matters
is the behaviour of certain functions (``signals'') distilled from
$\phi$. Adopting terminology used similarly in e.g.\ quantum mechanics
and ergodic theory \cite{Che, GS}, call any function $h:X\to \R$ an {\em
  observable\/} (on $X$). With this, what really matters from a
scientist's or engineer's point of view are properties of signals
$h\bigl( \phi(\bullet ,x)\bigr)$ for specific observables $h$ and
points $x\in X$ that are relevant to the process being modelled by $\phi$. In
the case of linear flows, a special role is naturally played by {\em
  linear\/} observables. Note that if $\phi$ and $h$ both are
linear then $h\bigl( \phi(t,x)\bigr)\equiv H(\phi_t)$, where $H :
\cL(X) \to \R$ is the linear functional with $H(A)=
h(A x)$ for all $A\in \cL(X)$. Given any linear flow $\phi$ on $X$, it
makes sense, therefore, to more generally consider signals $ H(\phi_{\bullet})$
where $H:\cL (X)\to \R$ is any linear functional; by
a slight abuse of terminology, such functionals will henceforth be referred to as {\em linear observables\/}
$\bigl($on $\cL(X)\bigr)$ as well.

What, if anything, can be said about the distribution of values for
signals $H(\phi_{\bullet})$, where $\phi$ and $H$ are a linear
flow on $X$ and a linear observable on $\cL(X)$, respectively? As indicated below
and demonstrated rigorously through the results of this article, for the overwhelming majority
of linear flows this question has a surprisingly simple, though perhaps
somewhat counter-intuitive answer: Except for the trivial case of $H(\phi_{\bullet})= 0$, that
is, $H(\phi_t)=0$ for all $t\in \R$, the values of $H(\phi_{\bullet})$ always exhibit {\em one and the same\/}
distribution, regardless of $d$, $\phi$ and $H$.
As it turns out, this distinguished distribution is nothing other than {\em Benford's Law\/} (BL), the logarithmic law for significant
digits.

Within the study of (digits of) numerical data generated by dynamical
processes --- a
classical subject that continues to attract interest
from disciplines as diverse as ergodic and number theory
\cite{ARS,CK,DK,KM, LagSou}, analysis \cite{BumEll, MasSch}, and
statistics \cite{diek,GW,MiNi} --- the astounding ubiquity of BL is a
recurring, popular theme. The most well-known special case of BL is the so-called ({\em decimal\/})
{\em first-digit law\/} which asserts that
\begin{equation}\label{eq1}
\mathbb{P} (\mbox{\em leading digit}_{10}\, = \ell ) = \log_{10} \left( 1 +
  \ell ^{-1} \right)  \quad \forall \ell  = 1 , \ldots , 9 \, ,
\end{equation}
where {\em leading digit}$_{10}$ refers to the leading (or first
significant) decimal digit, and $\log_{10}$ is
the base-$10$ logarithm (see Section \ref{sec2} for rigorous definitions); for example, the leading decimal
digit of $e =2.718$ is $2$, whereas the leading digit of $-e^{e}=-15.15$ is
$1$. Note that (\ref{eq1}) is heavily skewed towards the smaller
digits: For instance, the leading decimal digit is almost six
times more likely to equal $1$ (probability $\log_{10} 2=30.10$\%) than to equal $9$ (probability $1- \log_{10} 9 =4.57$\%).
Ever since first recorded by Newcomb \cite{newcomb} in 1881
and re-discovered by Benford \cite{benford} in 1938, examples of
data and systems conforming to (\ref{eq1}) in one form or another
have been discussed extensively, notably for real-life data (e.g.\
\cite{doC, sam}) as well as in stochastic (e.g.\ \cite{schuerg}) and
deterministic processes (e.g.\ the Lorenz flow \cite{TBL} and certain
unimodal maps \cite{BHPUP, SCD}). As of this writing, an online database \cite{BOB} devoted exclusively to BL lists more
than 800 references.

Given any (Borel) measurable function $f:\R^+ \to \R$, arguably the simplest and most
natural notion of $f$ conforming to (\ref{eq1}) is to require that
\begin{equation}\label{eq2}
\lim\nolimits_{T\to +\infty} \frac{\lambda (\{  t \le T:
  \mbox{\em leading digit}_{10} f(t) =\ell \}) }{T} = \log_{10}
\left(1+\ell ^{-1} \right)  \quad \forall \ell = 1, \ldots, 9 \, ; 
\end{equation}
here and throughout,
$\lambda$ denotes Lebesgue measure on $\R^+$, or on parts thereof.
With this, the central question studied herein is this: Does
(\ref{eq2}) hold for $f= H(\phi_{\bullet})$ where $\phi$ is a linear
flow on $X$ and $H$ is any linear observable on $\cL(X)$?
Several attempts to answer this question are recorded in the
literature; e.g., see \cite{BDCDSA, KNRS, NSh, TBL}. All these attempts, however, seem to
have led only to {\em sufficient\/} conditions for (\ref{eq2}) that are
either restrictive or complicated to state. In contrast, Theorem
\ref{thm32} below, one of the main results of this article, provides a
simple {\em necessary and sufficient\/} condition for every non-trivial
signal $f=H(\phi_{\bullet})$ to satisfy (\ref{eq2}), and in fact to
conform to BL in an even stronger sense. All results in the
literature alluded to earlier are but
simple special cases of this theorem.

To see why it is plausible for signals $f=H(\phi_{\bullet})$ to
satisfy (\ref{eq2}), pick any real number $\alpha\ne 0$ and consider as an extremely
simple but also quite compelling example the function $f(t) =
e^{\alpha t}$. Obviously, $x=f$ is a solution of $\dot x = \alpha x$, and $f(t)
\equiv \phi_t \in \cL(\R^1)$ for the linear flow generated by this
differential equation. A short elementary calculation shows that, for
all $T>0$ and $1\le \ell \le 9$,
\begin{equation}\label{eq2a}
\left|
\frac{\lambda\left( \left\{ t \le T:
  \mbox{\em leading digit}_{10} e^{\alpha t} =\ell \right\}\right) }{T} - \log_{10}
\left( 1+ \ell^{-1} \right) 
\right| < \frac{1 }{|\alpha| T} \, ,
\end{equation}
and hence (\ref{eq2}) holds for $f(t)=e^{\alpha t}$ whenever $\alpha
\ne 0$. (Trivially, it does not hold if $\alpha =0$.) However, already
for the linear flow $\phi$ on $\R^2$ generated by
$$
\dot x = \left[
\begin{array}{rr}
\alpha & - \beta \\
\beta & \alpha
\end{array}
\right] x \, ,
$$
with $\alpha, \beta \in \R$ and $\beta > 0$, a brute-force calculation
is of little use in deciding whether all non-trivial signals $f=H(\phi_{\bullet})$
satisfy (\ref{eq2}). Theorem \ref{thm32} shows that indeed they do, provided that $\alpha \pi/ (\beta \ln 10)$ is irrational; see
Example \ref{exa34}. 

This article is organized as follows. Section \ref{sec2} introduces
the formal definitions and analytic tools required for the
analysis. In Section \ref{sec3}, the main results characterizing
conformance to BL for linear flows are stated and
proved, based upon a tailor-made notion of exponential nonresonance (Definition
\ref{def2m2}). Several examples are presented in order to illustrate this
notion as well as the main results. Section \ref{sec4} establishes the
fact that, as suggested by the simple examples in the preceding paragraph, exponential
nonresonance, and hence conformance to BL as well, is generic for
linear flows on $\R^d$. Given the widespread use of 
linear differential equations as models throughout the sciences, the
results of this article may contribute to a better understanding of,
and deeper appreciation for, BL and its applications across a wide range of disciplines. 

\section{Definitions and tools}\label{sec2}

The following, mostly standard notation and
terminology is used throughout. The symbols $\N$, $\Z^+$, $\Z$, $\Q$,
$\R^+$, $\R$, and $\C$ denote the sets of, respectively, positive
integer, non-negative integer, integer, rational, non-negative real, real,
and complex numbers, and $\varnothing$ is the empty set. Recall that Lebesgue
measure on $\R^+$ or subsets thereof is written simply as $\lambda$. For each
integer $b\ge 2$, the logarithm base $b$ of $x>0$ is denoted
$\log_b x$, and $\ln x$ is the natural logarithm (base $e$) of $x$; for
convenience, let $\log_b0:=0$ for every $b$, and $\ln 0:= 0$. Given any $x\in \R$, the
largest integer not larger than $x$ is symbolized by $\lfloor x
\rfloor$. The real part, imaginary part, complex conjugate, and
absolute value (modulus) of any $z\in \C$ is $\Re z$, $\Im z$,
$\overline{z}$, and $|z|$, respectively. For each $z\in \C\setminus
\{0\}$, there exists a unique number $-\pi < \arg z \le \pi$ with $z =
|z|e^{\imath \arg z}$. Given any
$w\in \C$ and $\cZ\subset \C$, define $w + \cZ:= \{w+z : z\in \cZ\}$
and $w\cZ:= \{wz:z\in \cZ\}$. Thus with the unit circle $\mS:=\{ z\in \C : |z|=1\}$,
for example, $w+ \mS= \{z\in \C: |z-w|=1\}$ and $w\mS = \{z\in \C :|z| =|w|\}$ for each $w\in \C$.
The cardinality (number of
elements) of any finite set $\cZ\subset \C$ is $\#\cZ$.

Recall throughout that $b$ is an integer with $b\ge 2$, informally
referred to as a {\em base}. Given a base $b$ and any $x\ne 0$, there exists a unique
real number $S_b(x)$ with $1\le S_b(x)<b$ and a unique integer $k$ such that $|x|=
S_b(x)b^k$. The number $S_b(x)$ is the
{\em significand} or {\em mantissa} (base $b$) of $x$; for
convenience, define $S_b(0):=0$ for every base $b$. The integer $\lfloor
S_b(x) \rfloor$ is the {\em first significant digit\/} (base $b$) of
$x$; note that $\lfloor S_b(x)\rfloor \in \{1, \ldots , b-1\}$ whenever
$x\ne 0$. 

In this article, conformance to BL for real-valued functions,
specifically for signals generated by linear flows, is studied via the
following definition.

\begin{defn}
Let $b\in \N \setminus \{1\}$. A (Borel) measurable function $f:\R^+ \to \R$ is a $b\,$-{\em Benford function}, or
$b\,$-{\em Benford\/} for short, if
$$
\lim\nolimits_{T\to +\infty} \frac{\lambda \bigl( \bigl\{ t \le T  :
  S_b\bigl( f(t) \bigr) \le
  s \bigr\}\bigr)}{T} = \log_b s  \quad \forall s\in [1,b)  \, .
$$
The function $f$ is a {\em Benford function}, or simply {\em
  Benford}, if it is $b\,$-Benford for every $b\in \N \setminus \{1\}$.
\end{defn}

Note that (\ref{eq2}) holds whenever $f$ is
$10$-Benford. The converse is not true in general since, for
instance, the (piecewise constant) function $\lfloor S_{10}(2^{\bullet})
\rfloor$ only attains the values $1, \ldots, 9$ and hence clearly is
not $10$-Benford, yet (\ref{eq2a}) with $\alpha = \ln 2$ shows that it does
satisfy (\ref{eq2}).

The subsequent analysis of the Benford property for signals generated
by linear flows is greatly facilitated by a few basic facts from the
theory of uniform distribution, reviewed here for the reader's
convenience; e.g., see \cite{DT, KN} for authoritative accounts of the
subject. Throughout, the symbol $d$ denotes a positive integer, usually unspecified or clear
from the context. The $d$-dimensional torus $\R^d /\Z^d$ is symbolized by
$\T^d$, its elements being represented as $\langle x \rangle = x+\Z^d$
with $x\in \R^d$; for simplicity write $\T$ instead of $\T^1$. Endow
$\T^d$ with its usual (quotient) topology and $\sigma$-algebra $\cB(\T^d)$
of Borel sets, and let $\cP(\T^d)$ be the set of all probability
measures on $\bigl( \T^d , \cB (\T^d)\bigr)$. Denote the Haar
(probability) measure of the compact Abelian group $\T^d$ by
$\lambda_{\T^d}$. Call a set $\cJ\subset \T$ an {\em arc} if $\cJ=\langle
\cI \rangle := \{\langle x \rangle : x \in \cI\}$ for some interval
$\cI\subset \R$. With this, a (Borel) measurable function
$f:\R^+ \to \R$ is {\em continuously uniformly distributed modulo one}, henceforth
abbreviated {\em c.u.d.}\ mod $1$, if
$$
\lim\nolimits_{T\to +\infty} \frac{\lambda \bigl( \bigl\{   t  \le T : \langle f(t)
  \rangle \in \cJ\bigr\} \bigr)}{T} = \lambda_{\T}(\cJ) \quad \mbox{\rm for every
  arc $\cJ\subset \T$}\, .
$$
Equivalently, $\lim_{T\to +\infty} \frac1{T}
\int_0^T F\bigl(\langle f(t) \rangle\bigr)\, {\rm d}t = \int_{\T} F \, {\rm
  d}\lambda_{\T}$ for every continuous (or just Riemann
integrable) function $F:\T \to \C$. In particular, therefore, if the
function $f$
is c.u.d.\ mod $1$ then $\lim_{T\to +\infty} \frac1{T}\int_0^T e^{2\pi
\imath k f(t)}{\rm d}t =0$ for every $k\in \Z \setminus \{0\}$, and
the converse is also true (Weyl's criterion \cite[Thm.I.9.2]{KN}).

The importance of uniform distribution concepts for the present
article stems from the following fact which, though very simple, is
nevertheless fundamental for all that follows.

\begin{prop}\label{prop_dia}{\rm $\!\!$\cite[Thm.1]{Dia}}
Let $b\in \N\setminus\{1\}$. A measurable function $f:\R^+ \to \R$ is $b\,$-Benford if and
only if the function $\log_b|f|$ is c.u.d.\ {\rm mod} $1$.
\end{prop}

In order to enable the effective application of Proposition
\ref{prop_dia}, a few basic facts from the theory of
uniform distribution are re-stated here. In this context, the
following discrete-time analogue of continuous uniform distribution is
also useful: A sequence $(x_n)$ of real numbers, by definition, is {\em uniformly
  distributed modulo one\/} ({\em u.d.}\ mod $1$) if the (piecewise
constant) function $f = x_{1 + \lfloor \bullet \rfloor}$ is c.u.d.\ mod
$1$, or equivalently, if
$$
\lim\nolimits_{N\to \infty} \frac{\# \{  n \le N : \langle x_n
  \rangle \in \cJ\}}{N} = \lambda_{\T}(\cJ) \quad \mbox{\rm for every
  arc $\cJ\subset \T$}\, .
$$

\begin{lem}\label{lem200}
For each measurable function
$f:\R^+ \to \R$ the following are equivalent:
\begin{enumerate}
\item $f$ is c.u.d.\ {\rm mod} $1$;
\item If $g:\R^+ \to \R$ is measurable and
  $\lim_{t\to +\infty}\bigl( g(t) - f(t)\bigr)$ exists (in $\R$) then $g$ is c.u.d.\ {\rm mod} $1$;
\item $kf$ is c.u.d.\ {\rm mod} $1$ for every $k\in \Z \setminus
  \{0\}$;
\item $f + \alpha \ln t$ is c.u.d.\ {\rm mod} $1$ for every $\alpha
  \in \R$.
\end{enumerate}
\end{lem}

\begin{proof}
Clearly, (ii), (iii), and (iv) each implies
(i), and the converse is \cite[Exc.I.9.4]{KN},
\cite[Exc.I.9.6]{KN}, and \cite[Lem.2.8]{BDCDSA}, respectively.
\end{proof}

\noindent
The next result is a slight generalization of \cite[Thm.I.9.6(b)]{KN}.

\begin{lem}\label{lem230}
Let the function $f:\R^+ \to \R$ be measurable, and $\delta_0 > 0$. If, for some
measurable, bounded $F:\T\to \C$ and $z\in \C$,
$$
\lim\nolimits_{N\to \infty} \frac1{N}\sum\nolimits_{n=1}^N F \bigl(
\langle f (n \delta) \rangle \bigr) =  z  \quad \mbox{\rm for
  almost all } 0<\delta < \delta_0 \, ,
$$
then also
\begin{equation}\label{eqTU}
\lim\nolimits_{T\to +\infty} \frac{1}{T}\int\nolimits_0^T F \bigl(
\langle f(t)\rangle\bigr)\, {\rm d}t = z \, .
\end{equation}
In particular, if the sequence $\bigl(
f(n\delta)\bigr)$ is u.d.\ {\rm mod} $1$ for almost all $0<\delta <
\delta_0$ then $f$ is c.u.d.\ {\rm mod} $1$.
\end{lem}

\begin{proof}
For each $n\in \N$, let $z_n =
\int_{\delta_0(n-1)}^{\delta_0 n}  F \bigl( \langle f(t) \rangle \bigr)\, {\rm
  d}t$. By the Dominated Convergence Theorem,
$$
\lim\nolimits_{N\to \infty} \frac1{\delta_0}
\int\nolimits_{0}^{\delta_0} \frac1{N} \sum\nolimits_{n=1}^N   F \bigl(
\langle f (n \delta) \rangle \bigr) \, {\rm d} \delta = z \, .
$$
On the other hand,
\begin{align*}
\frac1{\delta_0} \int_0^{\delta_0} \frac1{N} \sum\nolimits_{n=1}^N   F \bigl(
\langle f (n \delta) \rangle \bigr) \, {\rm d} \delta & = \frac1{\delta_0N}
\sum\nolimits_{n=1}^N \frac1{n} \int_{0}^{\delta_0 n} F \bigl( \langle
f(t) \rangle \bigr)\, {\rm d}t \\[1mm]
& = \frac1{\delta_0 N} \sum\nolimits_{n=1}^N \frac1{n}
\sum\nolimits_{\ell =1}^n z_{\ell} \, ,
\end{align*}
and since the sequence $(z_n)$ is bounded, a well-known Tauberian
theorem \cite[Thm.92]{Har} implies that
\begin{align*}
z & = \lim\nolimits_{N\to
  \infty} \frac1{\delta_0 N} \sum\nolimits_{n=1}^N z_n =
\lim\nolimits_{N\to \infty} \frac1{\delta_0 N} \int_0^{\delta_0 N}
F \bigl( \langle f(t) \rangle \bigr)\, {\rm d}t \\[1mm]
& = \lim\nolimits_{T\to +\infty} \frac1{T} \int_0^T F \bigl( \langle
f(t) \rangle \bigr) \, {\rm d}t \, .
\end{align*}
The second assertion now follows immediately by considering
specifically the functions $F(\langle x \rangle)= e^{2\pi \imath kx}$ for
$k\in \Z$, together with Weyl's criterion.
\end{proof}

The following result pertains to very particular functions that map $\T^d$
into $\T$; such functions will appear naturally in the next
section. Concretely, let $p_1, \ldots , p_d\in \Z$ and $\alpha
\in \R \setminus \{0\}$, and consider the function
$$
P_u  :
\left\{
\begin{array}{ccl}
\T^d &  \to  & \T \, , \\
\langle x\rangle & \mapsto  &  \big\langle p_1 x_1 + \ldots + p_d x_d + \alpha \ln | u_1
\cos (2\pi x_1) + \ldots + u_d \cos (2\pi x_d)|\big \rangle   \, ;
\end{array}
\right.
$$
here $u\in \R^d\setminus \{0\}$ may be thought of as a parameter. (Recall the
convention that $\ln 0 = 0$.) Note that $P_u$ is measurable (in fact,
differentiable $\lambda_{\T^d}$-a.e.), and so each $\mu \in \cP
(\T^d)$ induces a well-defined element $\mu \circ P_u^{-1}$ of
$\cP(\T)$, via $\mu \circ P_u^{-1} (B):= \mu \bigl( P_u^{-1}
(B)\bigr)$ for all $B\in \cB(\T)$. It is easy to see that $\mu \circ
P_u^{-1}$ is absolutely continuous (w.r.t.\ $\lambda_{\T}$) whenever
$\mu$ is absolutely continuous (w.r.t.\ $\lambda_{\T^d}$). For
the purpose of this work, only the case $\mu = \lambda_{\T^d}$ is of
further interest. Observe that $\lambda_{\T^d}\circ P_u^{-1}$ is
equivalent to (i.e., has the same nullsets as)
$\lambda_{\T}$. Moreover, for $\cP(\T)$ endowed with the
topology of weak convergence, the Dominated Convergence Theorem implies
that the $\cP(\T)$-valued function $u\mapsto
\lambda_{\T^d} \circ P_u^{-1}$ is continuous on $\R^d \setminus \{0\}$, for any fixed $p_1, \ldots, p_d\in \Z$ and $\alpha \in \R
\setminus \{0\}$. The arguments in \cite[Sec.5]{BE-JDDE} show that this function is non-constant, as might be expected.

\begin{prop}{\rm $\!\!$\cite[Thm.5.4]{BE-JDDE}}\label{propnonuni}
Given $p_1, \ldots, p_d\in \Z$, $\alpha \in \R \setminus \{0\}$, and
any $\nu \in \cP(\T)$, there exists $u\in \R^d\setminus \{0\}$ such that $\lambda_{\T^d} \circ P_u^{-1}\ne \nu$.
\end{prop}

\begin{rem}
Specifically for the case $\nu = \lambda_{\T}$, it has been conjectured in \cite{BE-JDDE} that $\lambda_{\T^d}\circ
P_u^{-1} = \lambda_{\T}$ (if and) only if $\prod_{j: p_j \ne 0} u_j
=0$, and hence $\lambda_{\T^d} \circ P_u^{-1} \ne \lambda_{\T}$ for
{\em most\/} $u\in \R^d\setminus \{0\}$.
\end{rem}

The remainder of this section reviews tools and terminology concerning
certain elementary number-theoretical properties of sets $\cZ \subset
\C$. Specifically, denote by $\mbox{\rm span}_{\Q}\cZ$
the smallest subspace of $\C$ (over $\Q$) containing $\cZ$;
equivalently, if $\cZ\neq \varnothing$ then $\mbox{\rm span}_{\Q}\cZ$ is the set of all finite
{\em rational\/} linear combinations of elements of $\cZ$, i.e.,
$$
\mbox{\rm span}_{\Q}\cZ = \bigl\{  \rho_1 z_1 + \ldots + \rho_n z_n
: n \in \N , \rho_1, \ldots , \rho_n \in \Q , z_1, \ldots , z_n
\in \cZ \bigr\} \, ;
$$
note that $\mbox{\rm span}_{\Q}\varnothing=\{0\}$. With this terminology,
recall that $z_1, \ldots, z_L\in \C$ are $\Q$-{\em independent\/} if
$\mbox{\rm span}_{\Q}\{z_1, 
\ldots, z_L\}$ is $L$-dimension\-al, or equivalently if $\sum_{\ell
  =1}^{L} p_{\ell } z_{\ell }= 0$ with
integers $p_1, \ldots , p_L$ implies $p_1 = \ldots = p_L=0$. The
notion of $\Q$-independence is crucial for the distribution mod $1$ of
certain sequences and functions, and hence, via Proposition
\ref{prop_dia}, also for the study of BL. A simple but useful fact in
this regard is as follows.

\begin{prop}{\rm $\!\!$\cite[Lem.2.6]{BE-JDDE}}\label{propudisc}
Let $\vartheta_0, \vartheta_1, \ldots, \vartheta_d\in \R$, and assume
that the function $F:\T^d \to \C$ is continuous, and non-zero
$\lambda_{\T^d}$-almost everywhere. If the $d+2$ numbers $1,
\vartheta_0, \vartheta_1, \ldots, \vartheta_d$ are $\Q$-independent
then the sequence
$$
\Bigl(
n\vartheta_0 + \alpha \ln n + \beta \ln \big|
F\bigl(\langle (n \vartheta_1, \ldots , n\vartheta_d) \rangle \bigr) + z_n
\big|
\Bigr)
$$
is u.d. {\rm mod} $1$ for every $\alpha, \beta\in \R$ and every
sequence $(z_n)$ in $\C$ with $\lim_{n\to \infty} z_n = 0$.
\end{prop}

The following definitions of nonresonance and exponential nonresonance have been
introduced in \cite{BE-JDDE} and \cite{BHPUP}, respectively. As will
become clear in the next section, they owe their specific form to
Propositions \ref{prop_dia}, \ref{propnonuni}, and \ref{propudisc}.

\begin{defn}\label{def31}
Let $b\in \N\setminus \{1\}$. A non-empty set $\cZ \subset \C$ with
$|z|=r$ for some $r>0$ and all $z\in \cZ$, i.e.\ $\cZ\subset r\mS$, is
$b${\em -nonresonant\/} if the associated set 
%\begin{equation}\label{eq31}
$$
\Delta_{\cZ}:= \left\{
1 + \frac{\arg z - \arg w}{2\pi} : z,w \in \cZ
\right\} \subset \R 
%\end{equation}
$$
has the following two properties:
\begin{enumerate}
\item $\Delta_{\cZ} \cap \Q = \left\{ 1 \right\}$;
\item $\log_b r \not \in \mbox{\rm span}_{\Q} \Delta_{\cZ}$.
\end{enumerate}
An arbitrary set $\cZ\subset \C$ is $b$-nonresonant if, for
every $r>0$, the set $\cZ\cap r\mS$ is either $b$-nonresonant or empty;
otherwise, $\cZ$ is $b${\em -resonant}.
\end{defn}

\begin{defn}\label{def2m2}
Let $b\in \N\setminus \{1\}$. A set $\cZ \subset \C$ is {\em
  exponentially\/} $b${\em -nonresonant\/} if the set $e^{t\cZ}:=
\{e^{tz}: z\in \cZ\}$ is $b$-nonresonant for some $t\in \R^+$;
otherwise, $\cZ$ is {\em exponentially\/} $b${\em -resonant}.
\end{defn}

\begin{example}\label{ex2m1}
The empty set $\varnothing$ is $b$-nonresonant and exponentially
$b$-resonant for every $b$. The singleton $\{z\}$ with $z\in \C$ is $b$-nonresonant if and only if either $z=0$
or $\log_b |z| \not \in \Q$, and it is exponentially $b$-nonresonant
precisely if $\Re z \ne 0$. Similarly, any set $\{z, \overline{z}\}$
with $z\in \C \setminus \R$ is $b$-nonresonant if and only if $1$,
$\log_b |z|$ and $\frac1{2 \pi}\arg z$ are $\Q$-independent, and it is
exponentially $b$-nonresonant precisely if $\Re z \pi / ( \Im z \ln b)
\not \in \Q$.
\end{example}

Note that if $\cZ$ is (exponentially) $b$-nonresonant then so are the
sets $-\cZ := (-1)\cZ$ and $\overline{\cZ}:= \{\overline{z} : z\in
\cZ\}$, as well as every
$\cW \subset \cZ$. Also, for each $n\in \N$ the set
$\cZ^n:=\{z^n: z\in \cZ\}$ is $b$-nonresonant whenever $\cZ$ is. The converse
fails since, for instance, $\cZ = \{-e,e\}$ is
$b$-resonant whereas $\cZ^2 = \{e^2\}$ is $b$-nonresonant.
Similarly, if
$\cZ$ is exponentially $b$-nonresonant then so is $t\cZ$ for all $t\in
\R\setminus \{0\}$. On the other hand, a set $\cZ$ is certainly $b$-resonant if
$\cZ \cap \mS\ne \varnothing$, and it is exponentially $b$-resonant
whenever $\cZ \cap \imath \R \ne \varnothing$. 

The following simple observation establishes an alternative
description of exponential $b$-nonresonance. Recall that a set is {\em
countable\/} if it is either finite (possibly empty) or countably
infinite.

\begin{lem}\label{prop7char}
Let $b\in \N \setminus \{1\}$. Assume that the set $\cZ\subset \C$ is countable and
symmetric w.r.t.\ the real axis, i.e., $\overline{\cZ} =\cZ$. Then the following are equivalent:
\begin{enumerate}
\item $\cZ$ is exponentially $b$-nonresonant;
\item For every $z\in \cZ$,
\begin{equation}\label{eq830a}
\Re z \not \in \mbox{\rm span}_{\Q} \left\{  \frac{\ln b }{\pi} \Im w  : w \in \cZ , \Re w = \Re z \right\} \, .
\end{equation}
\end{enumerate}
Moreover, if {\rm (i)} and {\rm (ii)} hold then the set $\{t\in \R^+ :
e^{t\cZ} \enspace \mbox{\rm is $b$-resonant}\}$ is countable.
\end{lem}

\begin{proof}
To show (i)$\Rightarrow$(ii), suppose there exist different elements
$w_1, \ldots, w_L$ of $\cZ$ with $\Re w_1 = \ldots = \Re w_L$, as well
as $p_1, \ldots , p_L\in\Z$ and $q\in \N$ such that
$$
\Re w_1 = \sum\nolimits_{\ell = 1}^L \frac{p_{\ell}}{q} \frac{\ln
  b}{\pi }\Im w_{\ell} \, .
$$
Pick any $t>0$, let $r:= e^{t\Re w_1}$, and note that
$$
\log_b r = \frac{t \Re w_1}{\ln b}  =  \sum\nolimits_{\ell = 1}^L
\frac{p_{\ell}}{q} \frac{t \Im
  w_{\ell}}{\pi} \, . 
$$
On the other hand, since $\cZ$ is symmetric w.r.t.\ the real axis, and
since $\arg e^{t w_{\ell}}$ differs from $t\Im w_{\ell}$ by an integer
multiple of $2\pi$,
$$
\mbox{\rm span}_{\Q} \Delta_{e^{t\cZ} \cap r \mS } \supset \mbox{\rm
    span}_{\Q} \left( \{1\} \cup \left\{\frac{t\Im w_{\ell}}{\pi} :
      \ell = 1, \ldots , L \right\} \right) \, .
$$
Thus $\log_b r \in \mbox{\rm span}_{\Q} \Delta_{e^{t\cZ} \cap r \mS}$, showing that $e^{t\cZ}$ is $b$-resonant for all $t>0$. Since
  clearly $e^{0\cZ}=\{1\}$ is $b$-resonant as well, $\cZ$ is
  exponentially $b$-resonant, contradicting (i). Hence
  (i)$\Rightarrow$(ii); note that the countability of
  $\cZ$ has not been used here.

To establish the reverse implication (ii)$\Rightarrow$(i), suppose
the set $\cZ$ is exponentially $b$-resonant. In this case, for every $t>0$
there exists $r=r(t)>0$ such that $e^{t\cZ}\cap r \mS$ is
$b$-resonant, and so either $\Delta_{e^{t\cZ} \cap r\mS}\cap
\Q \ne \{1\}$ or $\log_b r \in \mbox{\rm span}_{\Q} \Delta_{e^{t\cZ} \cap
  r\mS}$, or both. In the first case, there exist elements
$w_1=w_1(t)$ and $w_2= w_2(t) $ of $\cZ$ with $\Re w_1 = \Re w_2$ but $w_1 \ne w_2$ such that $t(\Im w_1 -
\Im w_2)\in \pi \Q \setminus \{0\}$. In particular, therefore,
\begin{equation}\label{eqCC}
t \in \bigcup\nolimits_{z\in \cZ} \bigcup\nolimits_{w\in \cZ \setminus
\{z\}: \Re w = \Re z} \frac{\pi}{\Im w - \Im z} \Q \: = : \:  \Omega_1 \, .
\end{equation}
In the second case, for some positive integer $L=L(t)$ and some $w_1(t), \ldots,
w_{L}(t)\in \cZ$ with $\Re w_1 = \ldots = \Re w_L = t^{-1} \ln r$,
$$
\log_b r = \frac{t  \Re w_1}{\ln b} \in \mbox{\rm span}_{\Q}
\Delta_{e^{t\cZ} \cap r\mS} \subset  \mbox{\rm span}_{\Q} \left( \{1\}
  \cup \left\{ \frac{t\Im w_{\ell}}{\pi} : \ell = 1, \ldots ,
    L\right\}\right) \, .
$$
With the appropriate $p_0(t), p_1(t), \ldots , p_L(t)\in \Z$ and
$q(t)\in \N$, therefore,
\begin{equation}\label{eqCOUNT}
t   \left( q \Re w_1 - \sum\nolimits_{\ell = 1}^L p_{\ell} \frac{ \ln b}{\pi}
\Im w_{\ell}   \right) = p_0\ln b \, .
\end{equation}
Since $\cZ$ is countable, the set $\Omega_1$ in (\ref{eqCC}) is countable
as well. Consequently, if $\cZ$ is exponentially $b$-resonant then
(\ref{eqCOUNT}) must hold for all but countably many $t>0$. Hence there exist $t_2>t_1>0$ such that
$L(t_2)=L(t_1)$, $w_{\ell}(t_2) = w_{\ell}(t_1)$ and similarly $p_{\ell}(t_2) = p_{\ell}(t_1)$ for all $\ell = 1,
\ldots , L$, as well as $q(t_2) = q(t_1)$. This in turn implies
$$
\Re w_1 = \sum\nolimits_{\ell = 1}^L \frac{p_{\ell}}{q} \frac{\ln b}{\pi} \Im
  w_{\ell}  \, ,
$$
which clearly contradicts (\ref{eq830a}). For countable $\cZ$,
therefore, (ii) fails whenever (i) fails, that is,
(ii)$\Rightarrow$(i); note that the symmetry of $\cZ$ has
not been used here.

Finally, if (i) and (ii) hold, and if $e^{t\cZ}$ is $b$-resonant for some
$t>0$ then, as seen in the previous paragraph, either $t\in \Omega_1$ or else, by
(\ref{eqCOUNT}),
$$
\frac{\ln b}{t} \in \bigcup\nolimits_{z\in \cZ} \mbox{\rm span}_{\Q}
\left(
\{\Re z\} \cup \left\{ 
\frac{\ln b }{\pi} \Im w : w\in \cZ , \Re w = \Re z
\right\}
\right) \: = : \:  \Omega_2 \, .
$$
Since $\cZ$ is countable, so are $\Omega_1$ and $\Omega_2$, and hence $\{t\in
\R^+ : e^{t\cZ} \enspace \mbox{\rm is $b$-resonant}\}$ is countable as
well.
\end{proof}

\begin{rem}
The symmetry and countability assumptions are essential
in Lemma \ref{prop7char}. If $\cZ$ is not symmetric w.r.t.\ the real axis
then the implication (i)$\Rightarrow$(ii) may fail, as is seen e.g.\ for
$\cZ =\{1+ \imath \pi /\ln 10 \}$ which is exponentially
$10$-nonresonant by Example \ref{ex2m1}, yet
does not satisfy (ii) for $b=10$. Conversely, if $\cZ$ is uncountable then
(ii)$\Rightarrow$(i) may fail. To see this, simply take $\cZ = \R \setminus
\{0\}$ which satisfies (ii) for all $b$, and yet $e^{t\cZ}$ is $b$-resonant for
every $t\in \R^+$. 
\end{rem}

Deciding whether a set $\cZ\subset \C$ is $b$-resonant may be difficult
in practice, even if $\# \cZ = 2$. For example, it is unknown whether
$\{z\in \C : z^2 +2z +3=0 \}$ is $10$-resonant; see
\cite[Ex.7.27]{BHPUP}. In many situations of practical interest, the
situation regarding {\em exponential\/} $b$-resonance is much
simpler. Recall that a number $z\in \C$ is {\em algebraic\/} (over $\Q$) if
it is the root of some non-constant polynomial with integer
coefficients.

\begin{lem}\label{lem212}
Let $b\in \N\setminus \{1\}$. Assume every element of $\cZ \subset \C$
is algebraic. Then $\cZ$ is exponentially $b$-nonresonant if and only
if $\cZ \cap \imath \R = \varnothing$.
\end{lem}

\begin{proof}
The ``only if'' part is obvious since, as seen earlier, $\cZ \cap
\imath \R \ne \varnothing$ always renders the set $\cZ$ exponentially $b$-resonant. To prove the ``if''
part, suppose that $\cZ \cap \imath \R = \varnothing$ yet $\cZ$ is
exponentially $b$-resonant. Since all of its elements are algebraic,
the set $\cZ$ is countable. By Lemma \ref{prop7char} there exist $z_1,
\ldots , z_L\in \cZ$ with $\Re z_1 = \ldots = \Re z_L$, as well as
$p_1, \ldots , p_L \in \Z$ and $q\in \N$ such that
$$
\Re z_1 = \sum\nolimits_{\ell = 1}^L \frac{p_{\ell}}{q} \frac{ \ln b}{\pi} \Im
z_{\ell} \, .
$$
(Recall that the proof of the implication (ii)$\Rightarrow$(i) in that lemma does not
require $\cZ$ to be symmetric w.r.t.\ the real axis.) Since $\Re z_1
\ne 0$, it follows that
$$
\frac{\pi}{\ln b} = \sum\nolimits_{\ell = 1}^L \frac{p_{\ell}}{q}
\frac{\Im z_{\ell}}{\Re z_{1}} \, ,
$$
which in turn implies that $\pi/\ln b$ is algebraic. However, by the
Gel'fond--Schneider Theorem \cite[Thm.1.4]{wald}, the number $\pi/\ln b$ is {\em not\/} algebraic for
any $b\in \N \setminus \{1\}$. This contradiction shows that $\cZ$
cannot be exponentially $b$-resonant if $\cZ \cap \imath \R =
\varnothing$. 
\end{proof}

\section{Characterizing BL for linear flows on $\R^d$}\label{sec3}

Let $\phi$ be a linear flow on $X=\R^d$. Recall that there exists a
unique $A_{\phi}\in \cL(X)$ such that $\phi_t = e^{tA_{\phi}}$ for all
$t\in \R$. (The linear map $A_{\phi}$ is sometimes referred to as the {\em generator\/} of
$\phi$.) In fact, with $I_X\in \cL(X)$ denoting the identity
map, $A_{\phi} = \lim_{t \to 0} t^{-1} (\phi_t - I_X) =
\frac{{\rm d}}{{\rm d}t} \phi_t |_{t=0}$, and $\phi$ is simply the
flow generated by the (autonomous) linear differential equation
$\dot x = A_{\phi} x$. Since conversely $(t,x)\mapsto e^{tA}x$ defines, for each $A\in \cL(X)$, a linear flow on $X$
with generator $A$, there is a one-to-one correspondence between
the family of all linear flows on $X$ and the space $\cL(X)$. Thus it
makes sense to define the {\em spectrum\/} of $\phi$ as
$$
\sigma(\phi) := \sigma (A_{\phi}) = \{z\in \C : z\enspace \mbox{\rm is an
  eigenvalue of $A_{\phi}$} \} \, .
$$
Note that $\sigma(\phi)\subset \C$ is non-empty, countable (in fact,
finite with $\# \sigma(\phi)\le d$) and symmetric w.r.t.\ the real
axis. 

Recall that the symbol $H$ is used throughout to denote a linear
observable (i.e., a linear functional) on $\cL(X)$. For convenience,
let $\cO(X)$ be the space of all such observables, i.e., $\cO(X)$ is
simply the dual of $\cL(X)$, endowed with the usual topology. The following is a basic linear algebra observation.

\begin{lem}\label{lem31}
Let $\phi$ be a linear flow on $X$. Given any non-empty set $\cZ
\subset \sigma (\phi)$ and any vector $u\in \R^{\cZ}$, there exists $H\in \cO (X)$ such that
\begin{equation}\label{eqLIN}
H(\phi_t) = \sum\nolimits_{z\in \cZ} e^{t\Re z} u_z \cos (t \Im z)  \quad \forall t\in \R \, .
\end{equation}
\end{lem}

\begin{proof}
For every real
$z\in \cZ$, pick $v_{z} \in X \setminus \{0\}$ such that $A_{\phi} v_{z} =
z v_{z}$ and let $\widetilde{v}_{z}:= v_{z}$. For
every non-real $z \in \cZ$, pick $v_{z},
\widetilde{v}_{z}\in X \setminus \{0\}$ such that
$$
A_{\phi} v_{z} =  v_{z} \Re z  - \widetilde{v}_{z}
\Im z  \, , \quad
A_{\phi} \widetilde{v}_{z} = v_{z}  \Im z +
\widetilde{v}_{z}  \Re
z \,  ;
$$
note that $v_{z}, \widetilde{v}_{z}$ are linearly
independent. With this, for each $z\in \cZ$,
\begin{align}\label{eqphit}
\phi_t v_z & = e^{t\Re z} \bigl( v_{z} \cos (t\Im
z) - \widetilde{v}_{z} \sin (t\Im z) \bigr) 
\nonumber \\[-3mm]
& \qquad \qquad \qquad \qquad \qquad \qquad \qquad \qquad \quad \forall t \in \R \, .\\[-3mm]
\phi_t \widetilde{v}_z & = e^{t\Re z} \bigl( v_{z} \sin (t\Im
z) + \widetilde{v}_{z} \cos (t\Im z ) \bigr) \nonumber
\end{align}
For
each $z\in \cZ$, pick a linear functional
$h_z$ on $X$ such that $h_z (v_z)= h_z
(\widetilde{v}_z) = \frac12 u_z$. Using (\ref{eqphit}) it is easily
verified that $H\in \cO(X)$ given by
$$
H(A) = \sum\nolimits_{z\in \cZ} \bigl( 
h_z (A v_z) + h_z (A\widetilde{v}_{z})
\bigr) \quad \forall A \in \cL(X) \, ,
$$
does indeed satisfy (\ref{eqLIN}).
\end{proof}

As it turns out, the set $\sigma (\phi)$ controls the Benford property
for all signals $H(\phi_{\bullet})$. This is the first main result of
the present article.

\begin{theorem}\label{thm32}
Let $b\in \N \setminus \{1\}$. For each linear flow $\phi$ on $X$ the
following are equivalent:
\begin{enumerate}
\item The set $\sigma (\phi)$ is exponentially $b$-nonresonant;
\item For every $H\in \cO(X)$ either the function
  $H(\phi_{\bullet})$ is $b$-Benford, or $H(\phi_{\bullet})= 0$. 
\end{enumerate}
\end{theorem}

\noindent
The proof of Theorem \ref{thm32} makes use of a
discrete-time analogue established in \cite{BE-JDDE}.

\begin{prop}{\rm $\!\!$\cite[Thm.3.4]{BE-JDDE}}\label{prop33}
Let $b\in \N \setminus \{1\}$. For each invertible $A\in \cL(X)$ the
following are equivalent:
\begin{enumerate}
\item The set $\sigma (A)$ is $b$-nonresonant;
\item For every $H\in \cO(X)$ either the sequence
  $\bigl(\log_b |H(A^n)|\bigr)$ is u.d.\ {\rm mod} $1$, or $H(A^n)\equiv 0$.
\end{enumerate}
\end{prop}

\begin{proof}[Proof of Theorem \ref{thm32}]
To prove (i)$\Rightarrow$(ii), let $\sigma (\phi)$ be exponentially
$b$-nonresonant and fix any $H\in \cO(X)$. Note
that $\sigma (\phi_{\delta}) = e^{\delta \sigma (\phi)}$ is
$b$-nonresonant for all but countably many $\delta > 0$, by
Lemma \ref{prop7char}. Since $\phi_{\delta}$ is invertible for all $\delta>0$,
Proposition \ref{prop33} implies that either $H(\phi_{\delta}^n)=
H(\phi_{n\delta})\equiv 0$, or else the sequence $\bigl(\log_b
|H(\phi_{n\delta})|\bigr)$ is u.d.\ mod $1$. Let 
$$
\delta_0 := \inf \{ \delta > 0 : H(\phi_{n\delta}) \equiv 0\} \ge 0 \, ,
$$
with $\inf \varnothing := +\infty$. If $\delta_0 = 0$ then there
exists a sequence $(\delta_n)$ with $\delta_n \searrow 0$ and
$H(\phi_{\delta_n})\equiv 0$. Since $t\mapsto H(\phi_t)$ is analytic,
it follows that $H(\phi_{\bullet})= 0$. If, on the other hand, $\delta_0
> 0$ then, for almost all $0<\delta< \delta_0$,
the sequence $\bigl( f(n\delta)\bigr)$ is u.d.\ mod $1$, with the
measurable function $f = \log_b |H(\phi_{\bullet})|$. Hence by Lemma
\ref{lem230}, $f$ is c.u.d.\ mod $1$, i.e.,
$H(\phi_{\bullet})$ is $b$-Benford.

To prove (ii)$\Rightarrow$(i), let $\sigma (\phi)$ be exponentially
$b$-resonant. By Lemma \ref{prop7char}, there exists $z_1\in \sigma
(\phi)$ such that
\begin{equation}\label{eqSPACE}
\Re z_1 \in \mbox{\rm span}_{\Q} \left\{
\frac{\ln b }{\pi} \Im w  : w\in \sigma(\phi), \Re w = \Re z_1
\right\} \, .
\end{equation}
Let $L$ be the dimension of the $\Q$-linear space in
(\ref{eqSPACE}). 

If $L= 0$ then $z_1 = 0$, and picking any $v\in X
\setminus \{0\}$ with $A_{\phi} v =0$ yields $\phi_t v \equiv v$. With
any linear functional $h$ on $X$ that satisfies $h(v)=1$, and with the
linear observable $H$ defined as $H(A) = h(A v)$ for all $A\in \cL(X)$, therefore,
$H(\phi_t)\equiv 1$ is neither $b$-Benford nor zero, i.e., (ii)
fails. 

For $L \ge 1$, it is possible to choose $z_1, \ldots, z_L\in \sigma(\phi)$ with $\Re z_1
= \ldots = \Re z_L$ and $0\le \Im z_1 < \ldots < \Im z_L$ such that
the $L$ numbers
$\frac1{\pi} \Im z_1 , \ldots , \frac{1}{\pi} \Im z_{L}$ are
$\Q$-in\-de\-pen\-dent. By (\ref{eqSPACE}) there exist $p_1, \ldots ,
p_{L}\in \Z$ and $q\in \N$ such that
$$
\Re z_1 = \sum\nolimits_{\ell = 1}^{L} \frac{p_{\ell}}{2q} \frac{\ln b}{\pi} \Im
  z_{\ell}  \, .
$$
Use Proposition \ref{propnonuni}, with $d=L$ and $\alpha = q/ \ln
b$, to chose $u\in \R^L \setminus \{0\}$ such that $\nu :=
\lambda_{\T^{L}}\circ P_u^{-1}\ne \lambda_{\T}$, and use Lemma
\ref{lem31} to pick $H\in \cO(X)$ with
$$
H(\phi_t) = e^{t\Re z_1} \sum\nolimits_{\ell = 1}^{L} u_{\ell} \cos
(t \Im z_{\ell})  \quad \forall t \in \R \, .
$$
Since $t\mapsto H(\phi_t)$ is analytic and non-constant, the set
$\{t\in \R^+: H(\phi_t)=0\}$ is countable. For all but
countably many $\delta>0$, therefore, $H(\phi_{n\delta})\ne 0$ for all
$n\in \N$. Consequently, for
almost all $\delta>0$ and all $n\in \N$, 
\begin{align*}
\bigl\langle \log_b |H(\phi_{n\delta})^q|\bigr\rangle 
& = \left\langle
\frac{n q \delta \Re z_1}{\ln b}  + \frac{q}{\ln b} \ln \left|
\sum\nolimits_{\ell = 1}^{L} u_{\ell} \cos (n\delta \Im z_{\ell} ) 
\right| 
\right\rangle\\[1mm]
& = \left\langle
\sum\nolimits_{\ell = 1}^{L} p_{\ell} \frac{n\delta \Im
  z_{\ell}}{2\pi} + \frac{q}{\ln b} \ln \left|
\sum\nolimits_{\ell = 1}^{L} u_{\ell} \cos \left( 2\pi \frac{n\delta
  \Im z_{\ell}}{2\pi}\right)
\right|
\right\rangle \\[1mm]
& = P_u \left(
\left\langle \left(
\frac{n\delta \Im z_1}{2\pi} , \ldots , \frac{n\delta \Im z_{L}}{2\pi}
\right) \right\rangle
\right) \, .
\end{align*}
The $L+1$ numbers $1, \frac1{2\pi}\delta \Im z_1, \ldots
, \frac1{2\pi} \delta \Im z_{L}$ are $\Q$-independent for all but
countably many $\delta>0$, and whenever they are, the sequence
$\left( \left\langle \left( \frac{1}{2\pi} n\delta \Im z_1, \ldots ,
  \frac1{2\pi } n\delta \Im z_{L} \right) \right\rangle \right)$ is uniformly
distributed on $\T^{L}$; e.g., see \cite[Exp.I.6.1]{KN}. Since the
function $e^{2\pi \imath k P_u}$ is Riemann integrable for each $k\in \Z$,
it follows that for almost all $\delta>0$,
\begin{align*}
\frac1{N}\sum\nolimits_{n=1}^N e^{2\pi \imath k \log_b
  |H(\phi_{n\delta})^q|} & \: \: \: = \: \: \:  \frac1{N}\sum\nolimits_{n=1}^N e ^{2\pi
  \imath k P_u \bigl(\bigl\langle  (  n\delta \Im z_1 /(2\pi), \ldots , n\delta
  \Im z_{L}/(2\pi)) \bigr\rangle \bigr)} \\[1mm]
& \stackrel{N\to \infty}{\longrightarrow} \int_{\T^{L}} e^{2\pi \imath k
P_u}\, {\rm d}\lambda_{\T^{L}} = \int_{\T}e^{2\pi \imath k y}\,
{\rm d}\nu (y) \quad \forall k \in \Z \, . 
\end{align*}
Recall that $\nu \ne \lambda_{\T}$, so $\int_{\T} e^{2\pi \imath
  k^* y}\, {\rm d}\nu(y) \ne 0$ for some integer $k^* \ne 0$, and
Lemma \ref{lem230} shows that
$$
\lim\nolimits_{T\to +\infty} \frac1{T} \int_0^T e^{2\pi \imath k^*
  \log_b|H(\phi_t)^q|} \, {\rm d}t = \int_{\T} e^{2\pi \imath
  k^* y}\, {\rm d}\nu(y) \ne 0 \, .
$$
Thus $\log_b |H(\phi_{\bullet})^q|= q \log_b |H(\phi_{\bullet})|$ is not c.u.d.\ {\rm
  mod} $1$, and neither is $\log_b|H(\phi_{\bullet})|$, by Lemma
\ref{lem200}. In other words, $H(\phi_{\bullet})$ is not $b$-Benford
(and clearly $H(\phi_{\bullet})\ne 0$). Overall, (ii) fails whenever
(i) fails, that is, (ii)$\Rightarrow$(i).
\end{proof}

\begin{example}\label{exa34}
By utilizing Example \ref{ex2m1}, the examples mentioned in the Introduction are easily reviewed in the
light of Theorem \ref{thm32}.

(i) For the linear flow $\phi$ on $\R^1$ generated by the scalar
equation $\dot x = \alpha x$, the set $\sigma(\phi)=\{\alpha\}$ is
exponentially $b$-nonresonant if and only if $\alpha \ne 0$.

(ii) For the linear flow $\phi$ on $\R^2$ generated by
$$
\dot x = \left[
\begin{array}{cr} \alpha & - \beta \\
\beta & \alpha
\end{array}
\right] x \, ,
$$
with $\alpha, \beta \in \R$ and $\beta > 0$, the set $\sigma (\phi) =
\{\alpha \pm \imath \beta\}$ is exponentially $b$-nonresonant if and only
if $\alpha \pi / (\beta \ln b )\not \in \Q$, and whenever it is, $H(\phi_{\bullet})$ is $b$-Benford for every 
$H\in \cO(\R^2)$ unless $H(I_{\R^2}) = H\left(
\left[
\begin{array}{cr} 0 & -1 \\ 1 & 0 
\end{array}
\right]
\right)=0$; in the latter case, $H(\phi_{\bullet})=0$.
\end{example}

\begin{example}
Let $p\ne 0$ be any real polynomial, and $\alpha \in \R$. The function
$f(t)=p(t)e^{\alpha t}$ is $b$-Benford if and only if $\alpha\ne
0$. To see this, simply note that $f$ solves a linear differential
equation with constant coefficients and order equal to the
degree of $p$ plus one. Thus, $f=H(\phi)$ for the appropriate linear
observable $H$ and linear flow $\phi$ with $\sigma(\phi) =\{\alpha\}$,
and the claim follows from Theorem \ref{thm32}. 
\end{example}

\begin{example}\label{exa35}
Theorem \ref{thm32} remains valid if (ii) is required to hold more
generally for all observables on $\cL(X)$ of the form $p\circ H$,
where $p$ is any real polynomial with $p(0)=0$
 and $H\in \cO(X)$. To illustrate this, consider the linear flow $\phi$ generated on $X=\R^3$ by
\begin{equation}\label{eq3EX35}
\dot x =
\left[
\begin{array}{crc}
1 & -\pi  & 0 \\
\pi & 1 & 0 \\
0 & 0 & \alpha
\end{array}
\right] x \, ,
\end{equation}
with $\alpha = \ln 10 - \frac12 =1.802$. Clearly, $\sigma(\phi) =
\{1\pm \imath \pi, \alpha\}$ is exponentially $b$-nonresonant for all
$b\in \N \setminus \{1\}$. Taking for instance $p(x)=x^2$, the generalized
version of Theorem \ref{thm32} just mentioned implies that $[\phi_{\bullet}]_{j,k}^2$
is Benford or trivial for all $1\le j,k\le 3$. Note that by Lemma \ref{lem200},
$$
t \mapsto \sqrt{\sum\nolimits_{j,k=1}^3 [\phi_t]_{j,k}^2} = \sqrt{2 e^{2t} +
e^{2\alpha t} } = e^{\alpha t} \sqrt{ 1 +2 e^{-2 (\alpha - 1)t}}
$$
is Benford as well. Though this does not follow from even the
generalized theorem, it nevertheless suggests that $\| \phi_{\bullet} \|$ may also be
$b$-Benford for some or even all norms $\|\cdot\|$ on $\cL(X)$. In
fact, to guarantee the latter, exponential $b$-nonresonance
of an appropriate subset of $\sigma(\phi)$ suffices; see Theorem \ref{thm313}
below.

While $h(\phi_{\bullet})$ thus is Benford for {\em some\/}
  non-linear observables $h$ on $\cL(X)$ also, it should be noted that, on
the other hand, $h(\phi_{\bullet})$ may fail to be $b$-Benford even for
very simple polynomial observables $h$, despite $\sigma(\phi)$ being
exponentially $b$-nonresonant. Concretely, the implication
(i)$\Rightarrow$(ii) in Theorem \ref{thm32} fails if
the linear observable $H$ in (ii) is replaced by $h = p \circ H_1 + p
\circ H_2$ where $p$ is a real polynomial with $p(0)=0$, and $H_1,
H_2\in \cO(X)$. To see
this, let $\phi$ be again the linear flow on $\R^3$ generated by
(\ref{eq3EX35}), and take $p(x)=x^3$ as well as $H_1 = [\, \cdot \,
]_{1,1} - [\, \cdot \,]_{3,3}$ and $H_2 = [\, \cdot \, ]_{3,3}$. Then
\begin{align*}
h(\phi_t) & = ([\phi_t]_{1,1} - [\phi_t]_{3,3})^3 + [\phi_t]_{3,3}^3\\
& =
e^{(1+2\alpha)t} \cos (\pi t)
\left( 3 - 3 e^{-(\alpha - 1)t} \cos (\pi t) +
  e^{-2(\alpha - 1)t} \cos (\pi t)^2 \right)\quad \forall t \in \R \, ,
\end{align*}
and it is straightforward to see that $h(\phi_{\bullet})\ne 0$ is not $10$-Benford.

Finally, while the implication (i)$\Rightarrow$(ii) in Theorem
\ref{thm32} remains valid if the linear observable $H$ in (ii) is
replaced by $p\circ H$ for any polynomial $p$ with $p(0)=0$, this
implication may fail if, only slightly
more generally, $H$ is instead replaced by $\varphi \circ H$ where
$\varphi :\R \to \R$ is real-analytic with
$\varphi (0)=0$. For a simple example illustrating this with $\phi$ as above,
let $\varphi (x)=x^2/(1+x^2)$ and $H=[\, \cdot \, ]_{3,3}$. Then
$$
\varphi \circ H (\phi_t ) = \frac{[\phi_t]_{3,3}^2}{1 + [\phi_t]_{3,3}^2} =
\frac{1}{1+e^{-2\alpha t}} \quad \forall t \in \R \, ,
$$
and since $\lim_{t\to +\infty} \varphi \circ H(\phi_t) = 1$, clearly
$\varphi \circ
H(\phi_{\bullet})\ne 0$ is not $b$-Benford for any $b$.
\end{example}

By combining it with Lemma \ref{lem212}, Theorem \ref{thm32} can be
given a simpler form that applies in many situations of practical
interest. To this end, call a linear flow $\phi$ on $X$
{\em algebraically generated\/} if there exists a basis $v_1, \ldots ,
v_d$ of $X$ such that the (uniquely determined) numbers $a_{jk}\in \R$
with $A_{\phi}v_j = \sum_{k=1}^d a_{jk} v_k$ for all $j= 1, \ldots , d$ are all algebraic. In other words, all entries of the
coordinate matrix of $A_{\phi}$ relative to the basis $v_1, \ldots ,
v_d$ are algebraic numbers. Note that $\phi$
is algebraically generated if and only if $\sigma(\phi)$ consists of
algebraic numbers only. The following, then, is an immediate consequence of
Lemma \ref{lem212} and Theorem \ref{thm32}.

\begin{cor}\label{prop36}
For each algebraically generated linear
flow $\phi$ on $X$ the following are equivalent:
\begin{enumerate}
\item $\sigma (\phi) \cap \imath \R = \varnothing$;
\item For every $H\in \cO(X)$ either
  the function $H(\phi_{\bullet})$ is Benford, or $H(\phi_{\bullet})=0$.
\end{enumerate}
\end{cor}

\begin{rem}\label{rem37}
A linear flow $\phi$ with $\sigma(\phi)\cap \imath \R =\varnothing$ is
commonly referred to as {\em hyperbolic}; e.g., see
\cite{Amann}. Thus, an algebraically generated linear
flow exhibits the Benford--or--trivial dichotomy of Corollary
\ref{prop36}(ii) if and only if it is hyperbolic.
\end{rem}

\begin{example}
In order to decide whether $\sigma (\phi) \cap \imath \R =
\varnothing$, it is not necessary to explicitly determine
$\sigma (\phi)$. For instance, if $\phi$ is a linear flow on $\R^2$
then $\sigma (\phi) \cap \imath \R =  \varnothing$ if and only if
\begin{equation}\label{eq38a}
\mbox{\rm trace}\, A_{\phi} \det A_{\phi} \ne 0 \quad \mbox{\rm or}
\quad
\det A_{\phi} < 0 \, .
\end{equation}
For a concrete example, let $\alpha, \beta \in \R$ be algebraic
and consider the linear second-order equation
\begin{equation}\label{eq38b}
\ddot y + \alpha \dot y + \beta y = 0 \, .
\end{equation}
Since $y=H(\phi_{\bullet})$ with the appropriate $H\in \cO (\R^2)$ and
the linear flow $\phi$ on $\R^2$ generated by
$$
\dot x = 
\left[
\begin{array}{rr}
0 & 1 \\
-\beta & - \alpha
\end{array}
\right] x \, ,
$$
Corollary \ref{prop36}, together with (\ref{eq38a}), shows that
every solution $y\ne 0$ of (\ref{eq38b}) is Benford if and only if
$\alpha \beta \ne 0$ or $\beta < 0$, or equivalently, if and only if $(1+\alpha^2)|\beta|>\beta$.
\end{example}

To motivate the second main result of this section, Theorem
\ref{thm310} below, note that even if $\sigma(\phi)$ is exponentially $b$-resonant,
the signals $H(\phi_{\bullet})$ may nevertheless be $b$-Benford for some or in
fact for {\em most\/} linear observables $H$.

\begin{example}\label{exa39}
Let $\phi$ be the linear flow on $X=\R^2$ generated by $\dot x = Ax$ with
$$
A = 
\left[
\begin{array}{cc}
1 & 1 \\ 1& 1
\end{array}
\right]  \, .
$$
Since $\sigma(\phi)= \{0,2\}$ is exponentially $b$-resonant for every
$b\in \N \setminus \{1\}$, there exists a linear observable $H$ for
which $H(\phi_{\bullet})\ne 0$ is not $b$-Benford. A simple example is
$H=[\, \cdot \,]_{1,1} - [\, \cdot \,]_{1,2}$ yielding
$H(\phi_t)\equiv 1$. However, from the explicit formula
$$
\phi_t = {\textstyle \frac12} e^{2t} A - {\textstyle \frac12} (A-
2I_{\R^2}) \quad \forall t \in \R \, ,
$$
it is clear that $H(\phi_{\bullet})$ is Benford {\em unless\/}
$H(A)=0$. For {\em most\/} $H\in \cO(X)$,
therefore, $H(\phi_{\bullet})$ is Benford. On the other hand, for the
time-reversed flow $\psi$, i.e., for
$\psi_t \equiv \phi_{-t}$, a signal $H(\psi_{\bullet})$ can be
$b$-Benford {\em only\/} if $H( A - 2 I_{\R^2})=0$. Thus $H(\psi_{\bullet})$ is, for most $H\in \cO(X)$, neither $b$-Benford nor trivial.
\end{example}

To formalize the observation made in Example \ref{exa39}, recall first
that the space $\cO(X)$ can, upon choosing a basis, be identified with $\R^{d^2}$. In
particular, therefore, the notion of a property holding for (Lebesgue) {\em almost
  every\/} $H\in \cO(X)$ is well-defined and independent of the
choice of basis. Given any $A\in \cL(X)$, for each $z\in \sigma (A)$
define $k_z\ge 0$ to be the maximal integer for which  
$$
\mbox{\rm rank} (A - z I_X)^{k+1} < \mbox{\rm rank} (A - z
I_X)^{k} \quad  \mbox{\rm if } z \in \R \, ,
$$ 
and
$$
\mbox{\rm rank} (A^2 - 2\Re
z A + |z|^2 I_X)^{k+1} < \mbox{\rm rank} (A^2  - 2\Re z
A + |z|^2 I_X)^{k}
\quad  \mbox{\rm if }  z \in \C \setminus
\R\, . 
$$
Equivalently, $1\le k_z+1\le d$ is the size of the largest block
associated with the eigenvalue $z$ in the Jordan Normal Form
(over $\C$) of $A$. Denote by $(r_A,k_A)$ the (unique) element of
$\{(\Re z, k_z): z \in \sigma(A)\}$ that is maximal in the
lexicographic order on $\R\times \Z$, and define the {\em dominant spectrum\/} of $A$ as
$$
\sigma_{\rm dom} (A):= \{z\in \sigma (A) : \Re z = r_A, k_z = k_A\} \, .
$$
Thus $\sigma_{\rm dom}(A)\subset \sigma(A)$ consists of all right-most
eigenvalues of $A$ that have a Jordan block of maximal size
associated with them. As
it turns out, for every linear flow $\phi$ on $X$, the set
$\sigma_{\rm dom}(\phi):= \sigma_{\rm dom}(A_{\phi})$, though usually
constituting but a
small part of $\sigma(\phi)$, governs the Benford property of
$H(\phi_{\bullet})$ for most linear observables $H$.

\begin{theorem}\label{thm310}
Let $b\in \N \setminus \{1\}$. For each linear flow $\phi$ on $X$ the following are equivalent:
\begin{enumerate}
\item The set $\sigma_{\rm dom}(\phi)$ is exponentially $b$-nonresonant;
\item For almost every $H\in \cO(X)$ the function
  $H(\phi_{\bullet})$ is $b$-Benford.
\end{enumerate}
\end{theorem}

The proof of Theorem \ref{thm310} makes use of the following two
observations which are a direct analogue of Lemma
\ref{lem31} and an immediate consequence of
\cite[Lem.5.3]{BE-JDDE}, respectively. The routine verification of both assertions
is left to the reader.

\begin{lem}\label{prop312}
Let $\phi$ be a linear flow on $X$. Given any non-empty
set $\cZ \subset \sigma_{\rm dom} (\phi)$ and any vector $u\in \R^{\cZ}$,
there exists $H\in \cO(X)$ such that, with
$r=r_{A_{\phi}}\in \R$ and $k=k_{A_{\phi}}\in \{0,\ldots , d-1\}$,
$$
H(\phi_t) = e^{rt} t^k \sum\nolimits_{z\in \cZ} u_z \cos (t\Im z)
\quad \forall t \in \R \, .
$$
\end{lem}

\begin{lem}\label{prop312a}
Let $\Omega \subset \R^+$ be finite. For each function $f: \Omega \to
\C$ the following are equivalent:
\begin{enumerate}
\item $\lim_{t\to +\infty} \Re \left( \sum_{\omega \in \Omega} f(\omega)
  e^{\imath \omega t}\right)$ exists;
\item $f(\omega ) = 0$ for every $\omega \in \Omega \setminus \{0\}$.
\end{enumerate}
\end{lem}

\begin{proof}[Proof of Theorem \ref{thm310}]
For convenience, define $\sigma_{\rm dom}^+:= \{z\in \sigma_{\rm dom}(\phi) : \Im
z \ge 0\}$, and let $r= r_{A_{\phi}}$ and $k=k_{A_{\phi}}$; clearly,
the set $\sigma_{\rm dom}^+ \subset r + \imath \R$ is non-empty, and is
exponentially $b$-nonresonant if and only if $\sigma_{\rm dom}(\phi)$ is.
Let $L$ be the dimension of $\mbox{\rm
  span}_{\Q}\left\{
\frac{\ln b}{\pi} \Im z : z \in \sigma_{\rm dom}^+
\right\}$, and observe that $L=0$ if and only if $\sigma_{\rm
  dom}^+= \{r\}$. Recall that, as a
consequence of, for instance, the Jordan Normal Form Theorem, there
exists a family $U_z, V_z$ ($z\in \sigma_{\rm dom}^+$) in $\cL(X)$ such that
\begin{equation}\label{peq1}
\phi_t = e^{rt} t^k \left(
\sum\nolimits_{z\in \sigma_{\rm dom}^+} \bigl( 
U_z \cos (t\Im z) + V_z \sin (t\Im z)
\bigr) + G(t)
\right)  \quad \forall t > 0 \, ,
\end{equation}
where $G$ is continuous with $\lim_{t\to
  +\infty} G(t)=0$, and $V_r = 0$ in case $r\in \sigma_{\rm dom}^+$. Moreover, $U_z \ne 0$ or $V_z \ne 0$ in
(\ref{peq1}) for at least one $z\in \sigma_{\rm dom}^+$, since otherwise
$\lim_{t\to +\infty} e^{-rt} t^{-k} H(\phi_t)=0$ for every $H\in \cO(X)$, whereas Lemma \ref{prop312}
guarantees, for each $z\in \sigma_{\rm dom}^+$, the existence of an $H$ with
$ e^{-rt} t^{-k} H(\phi_t) \equiv  \cos (t\Im z)$, an obvious contradiction.

Consider first the case $L=0$. Here $\sigma_{\rm dom}^+ = \{r\}$,
and (\ref{peq1}) yields
$$
H(\phi_t) = e^{rt} t^k \bigl( H(U_r) + H\circ G(t) \bigr) \quad
\forall t > 0 \, ,
$$
where $U_r \ne 0$. If $H(U_r)\ne 0$ then, for all sufficiently large
$t>0$,
$$
\log_b |H(\phi_t)| = \frac{rt}{\ln b} + \frac{k}{\ln b}\ln t +
\frac1{\ln b} \ln |H(U_r) + H\circ G(t)| \, .
$$
Note that $\sigma_{\rm dom}=\{r\}$ is exponentially $b$-resonant if
and only if $r\ne 0$. Lemma \ref{lem200} shows that $\log_b
|H(\phi_{\bullet})|$ is c.u.d.\ mod $1$ if and only if $t \mapsto
rt/\ln b$ is. Lemma \ref{lem230} and Proposition \ref{propudisc} imply
that the latter is the case if $r\ne 0$, while it is obviously not the
case if $r= 0$. Provided that $H(U_r)\ne 0$, therefore,
$H(\phi_{\bullet})$ is $b$-Benford precisely if $r\ne 0$. Since $\{H :
H(U_r)=0\}$ is a nullset (in fact, a proper subspace) in $\cO(X)$, it
follows that (i)$\Leftrightarrow$(ii) whenever $L=0$.

It remains to consider the case $L\ge 1$. In this case, pick $z_1,
\ldots, z_L\in \sigma_{\rm dom}^+$ such that $\frac{\ln b}{\pi}\Im
z_1, \ldots , \frac{\ln b}{\pi}\Im z_L$ are $\Q$-independent; for
convenience, let $\cZ:= \{z_1, \ldots , z_L\}$. For every
$z\in \sigma_{\rm dom}^+ \setminus \cZ$ there exists an integer
$L$-tuple $p^{(z)}$, i.e., $p^{(z)}\in \Z^L$, such that
\begin{equation}\label{peq2}
\Im z = \sum\nolimits_{\ell = 1}^L \frac{p_{\ell}^{(z)}}{q} \Im
z_{\ell} = \frac{p^{(z)} \cdot (\Im z_1, \ldots , \Im z_L)}{q} \, ,
\end{equation}
with the appropriate $q\in \N$ independent of $z$; here $u \cdot
v$ denotes the standard inner product on $\R^L$, that is,
$u\cdot v= \sum_{\ell=1}^L u_{\ell} v_{\ell}$. In addition, for
every $1\le j,\ell \le L$ let
$$
p_j^{(z_{\ell})} := \left\{
\begin{array}{cll}
q & & \mbox{\rm if } j = \ell \, , \\
0 & & \mbox{\rm otherwise} \, ;
\end{array}
\right.
$$
with this, (\ref{peq2}) is valid for all $z\in \sigma_{\rm
  dom}^+$. Note that the set $\{\pm p^{(z)} : z\in \sigma_{\rm
  dom}^+\}\subset \Z^L$ contains at least $2\# \sigma_{\rm dom}^+ -1$
different elements, and hence $z\mapsto p^{(z)}$ is one-to-one.

With these ingredients, given any $H\in \cO(X)$, deduce from
(\ref{peq1}) that
\begin{align}\label{peq2aa}
H(\phi_{qt}) & = e^{rqt} (qt)^k \left(
\sum\nolimits_{z\in \sigma_{\rm dom}^+} \bigl( 
H(U_z) \cos (qt \Im z) + H(V_z) \sin (qt \Im z)
\bigr) + H \circ G(qt)
\right) \nonumber \\[-2mm]
& \quad \\[-3mm]
& = e^{rqt} (qt)^k \left(
F_H \left( \left\langle \left( \frac{t\Im z_1}{2\pi}, \ldots ,
      \frac{t\Im z_L}{2\pi}\right)\right\rangle \right) + H\circ G(qt)
\right) \quad \forall t > 0 \, , \nonumber
\end{align}
where the smooth function $F_H:\T^{L}\to \R$ is given by
$$
F_H  (\langle x \rangle)  = \Re \left(
\sum\nolimits_{z\in \sigma_{\rm dom}^+} \bigl( H(U_z) - \imath H(V_z)
\bigr) e^{2\pi \imath p^{(z)} \cdot x}
\right) \, .
$$
Note that $F_H =0$ only if $H(U_z)=H(V_z)=0$ for all $z\in \sigma_{\rm
dom}^+$, whereas otherwise the set $\{\langle x \rangle : F_H(\langle
x \rangle) = 0\}$ is a $\lambda_{\T^L}$-nullset. Thus, $F_H(\langle x\rangle)\ne 0$ for
$\lambda_{\T^L}$-almost all $\langle x\rangle\in \T^L$ if and only if
\begin{equation}\label{peq2a}
H(U_z)\ne 0 \enspace \mbox{\rm or} \enspace H(V_z)\ne 0 \enspace
\mbox{\rm for some } z\in \sigma_{\rm dom}^+ \, .
\end{equation}

To establish (i)$\Rightarrow$(ii), assume that $\sigma_{\rm
  dom}(\phi)$ is exponentially $b$-nonresonant, fix any $H\in \cO(X)$,
and let $f_H := H(\phi_{\bullet})$. Deduce from (\ref{peq2aa}) that, for
all $n\in \N$ and $\delta>0$,
$$
f_H(q n\delta) = e^{rq n\delta} q^kn^k\delta^k 
\left(
F_H \left( 
\left\langle
\left(
n \frac{\delta \Im z_1}{2\pi} , \ldots , n \frac{\delta\Im z_L}{2\pi}
\right)
\right\rangle
\right) + H \circ G (qn\delta)
\right) \, .
$$
Observe that the $L+2$ numbers $1, r q\delta / \ln b,
\frac1{2\pi} \delta \Im z_1 , \ldots , \frac1{2\pi} \delta \Im z_L$ are
$\Q$-independent for all but countably many $\delta>0$, and whenever they
are, $f_H(qn\delta)\ne 0$ for all sufficiently large $n$. Hence by
Proposition \ref{propudisc}, with $d=L$, $\vartheta_0 = r q \delta /\ln b$ and
$\vartheta_{\ell}= \frac1{2\pi} \delta \Im z_{\ell}$ for $\ell = 1, \ldots ,
L$, and with
$$
\alpha = \frac{k}{\ln b} \, , \quad
\beta = \frac1{\ln b} \, , \quad
F = q^k \delta^k F_H \, , \quad
(z_n) = \bigl( q^k \delta^k H \circ G(qn \delta )\bigr) \, ,
$$
the sequence $\bigl(\log_b |f_H(qn\delta)|\bigr)$ is u.d.\ mod $1$ for almost all
$\delta > 0$. Lemma \ref{lem230} shows that $\log_b|f_H|$ is c.u.d.\ mod
$1$, i.e., $f_H = H(\phi_{\bullet})$ is $b$-Benford. In summary, given
any $H\in \cO(X)$, the signal $H(\phi_{\bullet})$ is $b$-Benford
whenever (\ref{peq2a}) holds. Since $U_z \ne 0$ or $V_z \ne 0$ for at
least one $z\in \sigma_{\rm dom}^+$, the set
$$
\bigl\{H : H(U_z)= H(V_z) =0 \: \, \forall z\in \sigma_{\rm dom}^+ \bigr \} \: \subset \: \cO(X)
$$
is a nullset (in fact, a proper subspace) in $\cO(X)$. For (Lebesgue)
almost every $H\in \cO(X)$, therefore, $H(\phi_{\bullet})$ is $b$-Benford.

To prove (ii)$\Rightarrow$(i), assume that $\sigma_{\rm dom}(\phi)$ is
exponentially $b$-resonant. By Lemma \ref{prop7char}, there exist
integers $\widetilde{p}_1, \ldots , \widetilde{p}_L$ and
$\widetilde{q}\in \N$ such that
$$
r = \sum\nolimits_{\ell = 1}^L
\frac{\widetilde{p}_{\ell}}{\widetilde{q}} \frac{\ln b}{\pi}  \Im z_{\ell} \, .
$$
Use Proposition \ref{propnonuni} with $d=L$, $p_{\ell} = 2
\widetilde{p}_{\ell}$ for $\ell =1,\ldots, L$, and $\alpha = \widetilde{q}/\ln
b$ to choose $u^* \in \R^L \setminus \{0\}$ such that
$\lambda_{\T^L}\circ P_{u^*}^{-1}\ne \lambda_{\T}$, and use Lemma
\ref{prop312} to pick $H^*\in \cO(X)$ with
$$
H^*(\phi_t) = e^{rt} t^k \sum\nolimits_{\ell = 1}^L u_{\ell}^* \cos
(t\Im z_{\ell}) \quad \forall t \in \R \, .
$$
It follows that
$$
e^{-rqt}(qt)^{-k} H^*(\phi_{qt}) = \sum\nolimits_{\ell = 1}^L u_{\ell}^* \cos (qt \Im
z_{\ell})  = \Re \left( \sum\nolimits_{\ell=1}^L u_{\ell}^* e^{\imath q
  t \Im z_{\ell}}\right) \quad \forall t > 0 \, ,
$$
whereas (\ref{peq2aa}) yields
$$
e^{-rqt} (qt)^{-k} H^*(\phi_{qt}) = \Re \left( \sum\nolimits_{z\in \sigma_{\rm
    dom}^+}  \bigl( H^*(U_z) - \imath H^*(V_z) \bigr) e^{\imath q t
  \Im z}
\right)
+ H^*\circ G(qt) \, .
$$
Since $\lim_{t\to +\infty} H^*\circ G(qt)=0$ and $\Im z_{\ell}>0$ for
all $\ell = 1, \ldots , L$, Lemma \ref{prop312a} shows that
for each $z\in \sigma_{\rm dom}^+$, 
$$
H^*(U_z) = \left\{
\begin{array}{cll}
u_{\ell}^* & & \mbox{\rm if } z=z_{\ell} \\
0 & & \mbox{\rm otherwise }
\end{array}
\right.
\enspace \mbox{\rm and} \quad
H^*(V_z)=0 \, .
$$
Next, pick any $H\in \cO(X)$ that satisfies (\ref{peq2a}), and
consider the function $g_H(t):= \widetilde{q} \log_b |t^{-k}
H(\phi_t)|$ for $t>0$. It follows from (\ref{peq2aa}) that, for almost
all $\delta>0$ and all sufficiently large $n\in \N$,
\begin{align*}
\bigl \langle g_H(qn\delta) \bigr \rangle & = \bigl\langle \widetilde{q} \log_b
|(qn\delta)^{-k} H(\phi_{qn\delta})| \bigr\rangle \\[1mm]
& = \left\langle
qn\delta \frac{\widetilde{q}r}{\ln b} + \frac{\widetilde{q}}{\ln b}
\ln \left|
F_H \left(
\left\langle \left(
n\frac{\delta \Im z_1}{2\pi} \, , \ldots , n \frac{\delta \Im z_L}{2\pi}
\right) \right\rangle
\right) + H\circ G (q n\delta)
\right|
\right\rangle \\[1mm]
& = \left\langle
Q_H \left(
\left\langle \left(
n\frac{\delta \Im z_1}{2\pi} \, , \ldots , n \frac{\delta \Im z_L}{2\pi}
\right) \right\rangle
\right)  + y_n
\right\rangle \, ,
\end{align*} 
with the (measurable) function $Q_H : \T^L \to \T$ given by
$$
Q_H (\langle x \rangle) = \left\langle \sum\nolimits_{\ell = 1}^L 2
\widetilde{p}_{\ell} q  x_{\ell} + \frac{\widetilde{q}}{\ln
b} \ln |F_H (\langle x \rangle)| \right\rangle \, ,
$$
and with an appropriate sequence $(y_n)$ in $\R$ that satisfies
$\lim_{n\to \infty}y_n = 0$. The $L+1$ numbers $1, \frac1{2\pi }\delta \Im
  z_1 , \ldots , \frac1{2\pi} \delta\Im z_L$ are
$\Q$-independent for all but countably many $\delta > 0$, and whenever
they are,
$$
\frac1{N} \sum\nolimits_{n=1}^N e^{2\pi \imath k g_H (nq \delta)} \:
\stackrel{N\to \infty}{\longrightarrow} \: \int_{\T^L} e^{2\pi \imath
  k Q_H} \, {\rm d}\lambda_{\T^L} \quad \forall k \in \Z \, .
$$
By Lemma \ref{lem230}, this means that
$$
\lim\nolimits_{T\to +\infty} \frac1{T} \int_0^T e^{2\pi \imath k g_H (t)} \,
{\rm d}t = \int_{\T^L} e^{2\pi \imath k Q_H} \, {\rm
  d}\lambda_{\T^L}\quad 
\forall k\in \Z \, .
$$
Note that $Q_{H^*} = P_{u^*}\circ M_q$, with the map $M_q:\T^L \to
\T^L$ given by $M_q(\langle x \rangle) = \langle qx \rangle$. Observe
that $\lambda_{\T^L}\circ M_q^{-1} = \lambda_{\T^L}$, and recall that
$\lambda_{\T^L}\circ P_{u^*}^{-1}\ne \lambda_{\T}$, hence 
$$
\lambda_{\T^L} \circ Q_{H^*}^{-1} = (\lambda_{\T^L} \circ
M_q^{-1})\circ P_{u^*}^{-1} = \lambda_{\T^L}\circ P_{u^*}^{-1} \ne
\lambda_{\T} \, ,
$$
and so $\int_{\T^L} e^{2\pi \imath k^* Q_{H^*}}\, {\rm
  d}\lambda_{\T^L}\ne 0$ for some $k^* \in \Z \setminus \{0\}$.
By the Dominated Convergence
Theorem, the $\cP(\T)$-valued function $H\mapsto \lambda_{\T^L}\circ Q_H^{-1}$ is continuous on the
(non-empty open) set $\{H: \mbox{\rm (\ref{peq2a}) holds}\}\subset
\cO(X)$. Since (\ref{peq2a}) holds in particular with $H=H^*$,
$$
\lim\nolimits_{H\to H^*} \int_{\T^L} e^{2\pi \imath k^* Q_H} \, {\rm
  d} \lambda_{\T^L} = \int_{\T^L}
e^{2\pi \imath k^*Q_{H^*}}\, {\rm d}\lambda_{\T^L} \ne 0
\, ,
$$
and consequently, for every $H$ sufficiently close to $H^*$,
$$
\lim\nolimits_{T\to +\infty} \frac1{T} \int_0^T e^{2\pi \imath k^* g_H(t)}\,
{\rm d}t \ne 0 \, ,
$$
which in turn shows that $g_H$ is not c.u.d.\ mod $1$, and neither is
$\log_b |H(\phi_{\bullet})|$, by Lemma \ref{lem200}. In summary,
$H(\phi_{\bullet})$ is not $b$-Benford whenever $H$ is sufficiently
close to $H^*$. Consequently, the set $\{H: H(\phi_{\bullet}) \:
\mbox{\rm is not $b$-Benford}\}$ contains a non-empty open set, and
hence is not a nullset in $\cO(X)$. Thus (ii)$\Rightarrow$(i), and the
proof is complete.
\end{proof}

\begin{example}\label{ex312aa}
The observations made for the flows $\phi$ and $\psi$ on $\R^2$ in
Example \ref{exa39} are fully consistent with Theorem \ref{thm310}:
The set $\sigma_{\rm dom}(\phi)=\{2\}$ is exponentially
$b$-nonresonant for all $b$, while $\sigma_{\rm dom}(\psi)=\{0\}$ is
exponentially $b$-resonant. Hence $H(\phi_{\bullet})$ is Benford for
almost all $H\in \cO(\R^2)$ whereas $H(\psi_{\bullet})$ is not.
\end{example}

As indicated already in Example \ref{exa35}, the Benford property may
be of interest for some {\em non-linear\/} observables also. A simple
natural example are norms on $\cL(X)$.

\begin{theorem}\label{thm313}
Let $b\in \N \setminus \{1\}$ and $\|\cdot \|$ any norm on
$\cL(X)$. If $\sigma_{\rm dom}(\phi)$ is exponentially $b$-nonresonant for
the linear flow $\phi$ on $X$ then $\|\phi_{\bullet}\|$ is $b$-Benford.
\end{theorem}

\begin{proof}
Using the same notation as in the proof of Theorem \ref{thm310} above,
let $f(t):= \log_b t^{-k}\|\phi_t\|$ for all $t>0$, and deduce from
(\ref{peq1}) and (\ref{peq2}) that
$$
f(qt) = \frac{rqt}{\ln b} + \frac1{\ln b} \ln 
\left\|
E \left(
\left\langle
\left( \frac{t\Im z_1}{2\pi} \, , \ldots , \frac{t\Im z_L}{2\pi}
\right) \right\rangle
\right) + G(qt)
\right\| \quad \forall t> 0 \, ,
$$
where the smooth function $E:\T^L \to \cL(X)$ is given by
$$
E(\langle x \rangle) = \sum\nolimits_{z\in \sigma_{\rm dom}^+} \bigl(
U_z \cos (p^{(z)}\cdot x) + V_z \sin (p^{(z)}\cdot x) 
\bigr) \, .
$$
Recall that $U_z\ne 0$ or $V_z\ne 0$ for at least one $z\in
\sigma_{\rm dom}^+$, which in turn implies that $E(\langle x \rangle
)\ne 0$, and hence also $\|E(\langle x \rangle)\|\ne 0$, for
$\lambda_{\T^L}$-almost all $\langle x \rangle \in \T^L$. The argument
is now analogous to the one establishing (i)$\Rightarrow$(ii) in
Theorem \ref{thm310}: For all but countably many $\delta > 0$, the
$L+2$ numbers $1, rq\delta /\ln b , \frac1{2\pi} \delta \Im
  z_1, \ldots , \frac1{2\pi} \delta \Im z_L$ are
$\Q$-independent, and whenever they are, the sequence $\bigl(f(qn\delta)\bigr)$
is u.d.\ mod $1$ by Proposition \ref{propudisc}, with $d=L$, $\vartheta_0=
rq \delta /\ln b$ and $\vartheta_{\ell} = \frac1{2\pi} \delta \Im
z_{\ell}$ for $\ell = 1, \ldots , L$, as well as
$$
\alpha  = 0 \, , \quad
\beta = \frac1{\ln b} \, , \quad
F = \|E\| \, , \quad
(z_n) = \bigl(\|E_n + G(qn\delta)\| - \|E_n\| \bigr) \, ,
$$
where $E_n = E \left( \left\langle \left( \frac1{2\pi} n\delta \Im z_1
    , \ldots , \frac1{2\pi} n \delta \Im z_L\right) \right \rangle \right)$.
As before, it follows that $f$ is c.u.d.\ mod $1$, and so is $\log_b
\|\phi_{\bullet}\|$, i.e., $\|\phi_{\bullet}\|$ is $b$-Benford.
\end{proof}

Unlike in Theorems \ref{thm32} and \ref{thm310}, the converse in
Theorem \ref{thm313} is not true in general: The signal $\|\phi_{\bullet}\|$ may
be $b$-Benford even if $\sigma_{\rm dom}(\phi)$ is exponentially
$b$-resonant. In fact, as the next example demonstrates, except for
the trivial case of $d=1$, it is
impossible to characterize the Benford property of $\|\phi_{\bullet}\|$
solely in terms of $\sigma(\phi)$, let alone $\sigma_{\rm dom}(\phi)$.

\begin{example}\label{ex314}
Let $b=10$ for convenience and denote by $|\cdot|$ the Euclidean (or
spectral) norm on $\cL(X)$, induced by the standard Euclidean norm
$|\cdot |$ on $X$, i.e., $|A|= \max \{|Ax|: x\in X, |x|=1 \}$ with
$|x|=\sqrt{x\cdot x}$. 
Consider the linear flow $\phi$ on $X=\R^2$ generated by
$$
\dot x = \left[  \begin{array}{rr} 1 & -2\pi /\ln 10  \\ 2\pi /\ln 10 &
    1 \end{array}\right] x \, .
$$ 
Since $\sigma(\phi) = \sigma_{\rm dom} (\phi) = \{1\pm 2\imath \pi
/\ln 10 \}$ is exponentially
$10$-resonant, by Theorem \ref{thm310} the signal $H(\phi_{\bullet})$ fails
to be $10$-Benford for many (in fact, most) $H \in \cO(\R^2)$. 
To see this explicitly, note that $H(\phi_{\bullet})=0$ if and only
if 
\begin{equation}\label{eqex01}
H(I_{\R^2}) = 0 \quad \mbox{\rm and} \quad
H \left(
\left[
\begin{array}{cr}
0 & -1 \\
1 & 0 
\end{array}
\right]
\right) = 0  \, ,
\end{equation}
and otherwise, with the appropriate $\rho > 0$ and $0\le \eta < 1$,
$$
H(\phi_t) = e^t \rho \cos \bigl( 2\pi (t/\ln 10 - \eta)\bigr) \quad
\forall t \in \R \, .
$$
For all but countably many $\delta > 0$ and all sufficiently large $n\in \N$,
\begin{align*}
\langle  \log_{10} |H(\phi_{n\delta})|\rangle & = 
\bigl\langle n\delta / \ln 10 + \log_{10}\rho + \log_{10} \big|\cos
\bigl( 2\pi (n\delta /\ln 10 - \eta ) \bigr)
\big| \bigr\rangle \\
& = \bigl\langle P (\langle n \delta/\ln 10 - \eta \rangle ) + \eta +
\log_{10}\rho \bigr\rangle \, ,
\end{align*}
with the map $P:\T \to \T$ given by
$$
P(\langle x \rangle ) = \bigl\langle  x + \log_{10}
|\cos (2\pi x)|\bigr\rangle \, .
$$
Since the sequence $(n\delta/\ln 10 - \eta)$ is u.d.\ mod $1$ for all but countably
many $\delta >0$, and since, as is easily checked, $\lambda_{\T} \circ P^{-1}\ne
\lambda_{\T}$, Lemma \ref{lem230} shows
that $H(\phi_{\bullet})$ is not Benford. Whenever (\ref{eqex01})
fails, therefore, $H(\phi_{\bullet})$ is neither $10$-Benford nor
trivial. On the other hand, $|\phi_t|\equiv e^t$ is Benford. Thus
Theorem \ref{thm313} can not in general be reversed.

Consider now also the linear flow $\psi$ on $X$ generated by 
$$
\dot x = \left[  \begin{array}{rr} 1 & -4\pi /\ln 10  \\ \pi /\ln 10  &
    1 \end{array}\right] x \, .
$$ 
Note that $A_{\psi}$ and $A_{\phi}$ are similar, so $\sigma(\psi) =
\sigma (\phi)$ and also $\sigma_{\rm dom}(\phi) = \sigma_{\rm dom}(\psi)$. A short calculation confirms that 
$$
  |\psi_t| = \frac{e^t}{4} 
\sqrt{25 - 9 \cos (4\pi t /\ln 10 ) + 3|\sin (2\pi t /\ln 10 )| \sqrt{82 -
  18 \cos (4\pi t /\ln 10 )}} \, ,
$$
and hence, for any $\delta > 0$ and $n\in \N$, 
$$
\langle \log_{10}  |\psi_{n\delta}|\rangle
= \bigl\langle Q(\langle  n \delta /\ln 10 \rangle) - \log_{10} 4
\bigr\rangle \, ,
$$
with the (piecewise smooth) map $Q:\T \to \T$ given by
$$
Q(\langle x \rangle ) = \left\langle  x + {\textstyle \frac12} \log_{10}\bigr(25 - 9
  \cos (4\pi x) + 3|\sin (2\pi x)| \sqrt{82 - 18 \cos (4\pi x)}\bigr) 
  \right\rangle \, .
$$
As before, it is straightforward to see that $\lambda_{\T} \circ Q^{-1} \ne
\lambda_{\T}$, and Lemma \ref{lem230} implies that $|\psi_{\bullet}|$
is not $10$-Benford. In summary, even though the linear flows $\phi$
and $\psi$ have identical spectra and dominant spectra,  the signal $|\phi_{\bullet}|$ is $10$-Benford whereas
the signal $|\psi_{\bullet}|$ is not.
\end{example}

\begin{rem}\label{rem315}
From Examples \ref{ex312aa} and \ref{ex314}, it may be conjectured
that if $\sigma_{\rm dom}(\phi)$ is exponentially $b$-resonant then
$\{H : H(\phi_{\bullet}) \: \mbox{\rm is $b$-Benford}\}$ actually is a
nullset in $\cO(X)$. By means of a stronger variant of Proposition
\ref{propnonuni} established in \cite{BE-JDDE}, it is not difficult to
see that this is indeed the case for $1\le d \le 4$. However, the
author does not know of a proof of, or counter-example to, this
conjecture for $d\ge 5$.
\end{rem}

\section{Most linear flows are Benford}\label{sec4}

As seen in the previous section, if $\sigma(\phi)$ is exponentially
nonresonant for the linear flow $\phi$ on $X= \R^d$, then the
Benford-or-trivial dichotomy of Theorem \ref{thm32}(ii) holds for every
signal $H(\phi_{\bullet})$. In fact, $H(\phi_{\bullet})$ is Benford
unless
\begin{equation}\label{eq41}
H(A_{\phi}^j) = 0 \quad \forall j = 0, \ldots , d -1 \, ;
\end{equation}
here, as usual, $A^0:= I_X$ for all $A\in \cL (X)$. Note that the
linear observables satisfying (\ref{eq41}), and hence
$H(\phi_{\bullet})=0$, constitute a proper subspace of $\cO(X)$. In fact, as seen in the proof of Theorem
\ref{thm310}, even if exponential nonresonance holds only for
$\sigma_{\rm dom}(\phi)$, the signal $H(\phi_{\bullet})$ is still
Benford, provided that $H$ does not belong to one distinguished proper
subspace of $\cO(X)$ that is independent of $H$. Put differently, if
$\sigma_{\rm dom}(\phi)$ or even $\sigma(\phi)$ is exponentially
nonresonant then BL is the only relevant digit distribution that can
be distilled from $\phi$ by means of linear observables. The purpose
of this short section is to demonstrate in turn that $\sigma(\phi)$,
and hence also $\sigma_{\rm dom}(\phi)$, is exponentially
$b$-nonresonant for all $b\in \N \setminus \{1\}$ and most linear
flows $\phi$, both from a topological and a measure-theoretical point of
view.

Recall that every linear flow $\phi$ on $X$ can be identified, via
$\phi \leftrightarrow A_{\phi}$ with a unique element of $\cL(X)$. The
latter space has a natural linear and topological structure making it
isomorphic and homeomorphic to $\R^{d^2}$, and hence it will be
convenient to phrase the results of this section as statements
regarding $\cL(X)$. Specifically, for every $b\in \N \setminus \{1\}$
consider the set of linear maps
$$
\cR_b:= \bigl\{A \in \cL(X) : \sigma(A) \: \mbox{\rm is exponentially
  $b$-resonant} \bigr\} \, ;
$$
also let $\cR:= \bigcup_{b\in \N \setminus \{1\}}\cR_b$. Recall that a
subset of a topological space is {\em meagre\/} (or {\em of first
  category\/}) if it is the countable union of nowhere dense
sets. According to the Baire Category Theorem, in a complete metric
space (such as, e.g., $\cL(X)$ endowed with any norm), meagre sets are, in a
sense, topologically negligible. The goal of this section, then, is to establish the
following fact which, informally put, shows that $\cR$ is a
negligible set, both topologically and measure-theoretically.

\begin{theorem}\label{thm42}
The set $\cR$ is a meagre nullset in $\cL(X)$.
\end{theorem}

A crucial ingredient in the proof of Theorem \ref{thm42} presented
below is
the real-analyticity of certain functions. Recall that a function $f:\cU \to \C$,
with $\cU\ne \varnothing$ denoting a connected open subset of $\R^L$ for some
$L\in \N$, is {\em real-analytic\/} (on $\cU$) if it can be, in a neighbourhood of
each point of $\cU$, represented as a convergent power series. An
important property of real-analytic functions not shared by arbitrary $\C$-valued
$C^{\infty}$-functions on $\cU$ is the following fact regarding their
zero-locus, which apparently is part of analysis folklore; e.g., see \cite[p.83]{KP}.

\begin{prop}\label{prop43}
Let $f:\cU\to \C$ be real-analytic, and $N_f:=\{x\in \cU: f(x)=0\}$. Then
either $N_f = \cU$, or else $N_f$ is a (Lebesgue) nullset.
\end{prop}

Next consider any monic polynomial $p_{a}:\C \to \C$ of degree
$L\ge 2$, i.e.,
$$
p_{a}(z) = z^L + a_{1}z^{L-1} + \ldots + a_{L-1}z +
a_L \, ,
$$
where $a=(a_1, \ldots , a_L)\in \R^L$, and recall that
$p_{a}$ has, for most $a \in \R^L$, only simple roots. More
formally, there exists a non-constant
real-analytic function $g_L:\R^L \to \R$ with the property that
if $p_{a}$ has a multiple root, i.e., $p_{a}(z_0) =
p_{a}'(z_0)=0$ for some $z_0\in \C$, then $g_L(a)=0$. In
fact, the function $g_L$ can be chosen as a polynomial with integer
coefficients and degree $2L-2$; e.g., see \cite[Lem.3.3.4]{Cohen}.
Whenever $g_L(a)\ne 0$, therefore, the equation $p_{a}(z)=0$
has exactly $L$ different solutions which, by (the real-analytic
version of) the Implicit Function Theorem \cite[Thm.2.3.5]{KP} depend
real-analytically on $a$. To put these facts together in a form
facilitating a proof of Theorem \ref{thm42}, for every $A_0 \in
\cL(X)$ and $\varepsilon > 0$, denote by $B_{\varepsilon}(A_0)$ the
open ball with radius $\varepsilon$ centered at $A_0$, that is,
$B_{\varepsilon}(A_0)=\{A \in \cL(X) :\|A- A_0\|< \varepsilon\}$,
where $\|\cdot \|$ is any fixed norm on $\cL(X)$.

\begin{lem}\label{lem44}
There exists a closed nullset $\cN \subset \cL(X)$ with the following
property: For each $A_0 \in \cL(X) \setminus \cN$ there exist
$\varepsilon > 0$ and $d$ real-analytic functions $\lambda_1 , \ldots
,\lambda_d: B_{\varepsilon}(A_0)\to \C$ such that, for all $A\in
B_{\varepsilon}(A_0)$,
\begin{enumerate}
\item $\sigma(A)=\{\lambda_1 (A), \ldots , \lambda_d (A)\}$;
\item $\lambda_j(A) \ne \lambda_k (A)$ whenever $j\ne k $;
\item $\lambda_j(A)\ne \overline{\lambda_k (A)}$ whenever $j\ne k$,
  unless $\lambda_j = \overline{\lambda_k}$ on $B_{\varepsilon}(A_0)$.
\end{enumerate}
\end{lem}

\begin{proof}
For $d=1$ simply take $\cN = \varnothing$ and $\lambda_1 ([a])=a$. For
$d\ge 2$, note that
$$
p_A(z) := \det (zI_X - A) = z^d + a_1(A)z^{d-1} + \ldots +
a_{d-1}(A) z + a_d (A) \, ,
$$
with real-analytic (in fact, polynomial) functions $a_1, \ldots ,
a_d : \cL(X)\to \R$; for example, $a_1(A) = - \mbox{\rm trace}\, A$ and
$a_d(A)=(-1)^d \det A$. Thus the function $g:= g_d(a_1, \ldots,
a_d):\cL (X) \to \R$ is real-analytic and non-constant, and so
$$
\cN:= \bigl\{ A\in \cL(X) : A \: \mbox{\rm has a multiple eigenvalue}\bigr\} =
\{A \in \cL(X) : g(A) = 0\}
$$
is a closed nullset, by Proposition \ref{prop43}. For each $A_0 \in
\cL(X)\setminus \cN$ there exists $\varepsilon > 0$ such that
$B_{\varepsilon}(A_0)\cap \cN = \varnothing$, and for $\varepsilon$
sufficiently small, by the Implicit
Function Theorem, there also exist $d$
real-analytic functions $\lambda_1, \ldots , \lambda_d :
B_{\varepsilon}(A_0)\to \C$ with $\sigma(A) = \{\lambda_1(A),
\ldots , \lambda_d(A)\}$ for all $A\in B_{\varepsilon}(A_0)$. Clearly,
$\lambda_j (A) \ne \lambda_k (A)$ whenever $j\ne k$, since otherwise
$g(A)=0$. Finally, if $\lambda_j(A_1) = \overline{\lambda_k(A_1)}$ for
some $A_1 \in B_{\varepsilon}(A_0)$ then $\overline{\lambda_k (A)}$ is, for every
$A$ sufficiently close to $A_1$, an eigenvalue of $A$ that, by
continuity, must coincide with $\lambda_j(A)$. Hence
$\lambda_j(A)=\overline{\lambda_k (A)}$ for all $A$ close to $A_1$,
and therefore, by Proposition \ref{prop43}, for all $A\in
B_{\varepsilon}(A_0)$ as well.
\end{proof}

\begin{proof}[Proof of Theorem \ref{thm42}]
Since for $d=1$ clearly $\cR_b = \{0\}$ for all $b$, the set $\cR=\{0\}$ is a meagre nullset in $\cL(X) =
\R$, and only the case $d\ge 2$ has to be considered henceforth. Fix $b\in \N \setminus \{1\}$, and
given any $p=(p_1, \ldots , p_d)\in \Z^d$, $q\in \N$, and non-empty
set $J\subset \{1, \ldots , d\}$, define a real-analytic (in fact,
polynomial) function $f_{p,q,J}:\R^{2d}\to \R$ as
$$
f_{p,q,J} (x) :=
\sum\nolimits_{j,k\in J} (x_j - x_k)^2 + \sum \nolimits_{j\in J}
\left(
\pi q x_j - \ln b \sum\nolimits_{k\in J} p_k x_{d+k}
\right)^2 \, .
$$
For each $A_0\in \cL(X)\setminus \cN$, pick $\varepsilon > 0$ and
$\lambda_1, \ldots , \lambda_d: B_{\varepsilon}(A_0)\to \C$ as in
Lemma \ref{lem44}. Observe that if $\sigma(A)$ is exponentially
$b$-resonant for some $A\in B_{\varepsilon}(A_0)$ then
\begin{equation}\label{eq45}
F_{p,q,J}(A):= f_{p,q,J} \bigl(
\Re \lambda_1 (A) , \ldots , \Re \lambda_d(A), \Im \lambda_1 (A) ,
\ldots , \Im \lambda_d(A)
\bigr) = 0 
\end{equation}
for the appropriate $p$, $q$, and $J$.
Clearly, every function $F_{p,q,J}: B_{\varepsilon}(A_0)\to \R$ is
real-analytic. Moreover, if
$F_{p,q,J}(A_1)=0$ for some $A_1\in B_{\varepsilon}(A_0)$ then also $A_1
+\delta I_X \in B_{\varepsilon}(A_0)$ for all sufficiently small
$\delta>0$, and $F_{p,q,J}(A_1 +\delta I_X)= \pi^2 q^2 \delta^2 \# J
>0$. Thus $F_{p,q,J}\ne 0$, and hence the set
$$
\cN_{p,q,J,A_0}:= \bigl\{
A\in B_{\varepsilon}(A_0) : F_{p,q,J}(A) = 0
\bigr\}
$$
is a closed nullset, by Proposition \ref{prop43}; in particular,
$\cN_{p,q,J,A_0}$ is nowhere dense, and (\ref{eq45}) implies that
$$
\cR_b \cap B_{\varepsilon}(A_0) = \bigcup\nolimits_{p,q,J}
\cN_{p,q,J,A_0} =: \cN_{A_0}\, .
$$
Being the countable union of nowhere dense nullsets, the set
$\cN_{A_0}$ is itself a meagre nullset. Since $\cL(X)$ is separable,
there exists a sequence $(A_{0,n})$ in $\cL(X)\setminus \cN$ and a sequence
$(\varepsilon_n)$ in $\R$ with $\varepsilon_n > 0$ for all $n$, such
that
$$
\cL(X) \setminus \cN = \bigcup\nolimits_{n\in \N} B_{\varepsilon_n}
(A_{0,n}) \, .
$$
It follows that
$$
\cR_b \subset \cN \cup \bigcup\nolimits_{n\in \N} \bigl( \cR_b \cap
B_{\varepsilon_n}(A_{0,n})\bigr) = \cN \cup \bigcup \nolimits_{n\in \N}
\cN_{A_0,n} \, ,
$$
which shows that $\cR_b$ is a meagre nullset as well, and so is $\cR =
\bigcup_{b\in \N \setminus \{1\}}\cR_b$.
\end{proof}

\begin{rem}\label{rem45}
Despite being a meagre nullset, the set $\cR$ could nevertheless be
dense in $\cL(X)$. This, however, is not the case: Lemmas
\ref{prop7char} and \ref{lem44} imply that $\cL(X)\setminus \cR$ contains the
non-empty open set $\bigl\{ A \in \cL(X)\setminus \cN : \sigma(A)
\subset  \R \setminus \{0\}\bigr\}$.
\end{rem}

Informally put, Theorems \ref{thm32} and \ref{thm42} together show that for a generic linear flow $\phi$ on
$X=\R^d$, the set $\sigma(\phi)$ is exponentially $b$-nonresonant for all
bases $b$, and so for each linear observable $H$ on $\cL(X)$ the signal
$H(\phi_{\bullet})$ is Benford unless (\ref{eq41}) holds, in which
case $H(\phi_{\bullet})=0$. This may provide yet another
explanation as to why BL is so often observed for even the simplest dynamical models 
in science and engineering.

\subsubsection*{Acknowledgements}

The author was supported by an {\sc Nserc} Discovery
Grant. He wishes to thank T.P.\ Hill, B.\ Schmuland, M.\ Waldschmidt, A.\ Weiss, and
R.\ Zweim\"{u}ller for many helpful discussions and comments which, in one
way or another, have informed this work.

\end{document}